\documentclass[leqno,11pt]{amsart}

\usepackage{amssymb, amsmath}
\usepackage{amssymb,amsfonts}
\usepackage{amsmath,latexsym}
\usepackage{pdfsync}
\usepackage{bbm}

\newtheorem{thm}{Theorem}[section]

\newtheorem{lem}{Lemma}[section]
\newtheorem{defi}{Definition}[section]

\newtheorem {Remark} {Remark} [section]

\numberwithin{equation}{section}

\makeatletter
\def\sommaire{\@restonecolfalse\if@twocolumn\@restonecoltrue\onecolumn
\fi\chapter*{Sommaire\@mkboth{SOMMAIRE}{SOMMAIRE}}
  \@starttoc{toc}\if@restonecol\twocolumn\fi}
\makeatother

\makeatletter
\def\thebibliographie#1{\chapter*{Bibliographie\@mkboth
  {BIBLIOGRAPHIE}{BIBLIOGRAPHIE}}\list
  {[\arabic{enumi}]}{\settowidth\labelwidth{[#1]}\leftmargin\labelwidth
  \advance\leftmargin\labelsep
  \usecounter{enumi}}
  \def\newblock{\hskip .11em plus .33em minus .07em}
  \sloppy\clubpenalty4000\widowpenalty4000
  \sfcode`\.=1000\relax}

\makeatother

\makeatletter
\def\references#1{\section*{R\'ef\'erences\@mkboth
  {R\'EF\'ERENCES}{R\'EF\'ERENCES}}\list
  {[\arabic{enumi}]}{\settowidth\labelwidth{[#1]}\leftmargin\labelwidth
  \advance\leftmargin\labelsep
  \usecounter{enumi}}
  \def\newblock{\hskip .11em plus .33em minus .07em}
  \sloppy\clubpenalty4000\widowpenalty4000
  \sfcode`\.=1000\relax}

\makeatother

\def\inte#1{
\displaystyle\mathop{#1\kern0pt}^\circ
}

\newcommand{\beq}{\begin{equation}}
\newcommand{\eeq}{\end{equation}}
\newcommand{\ben}{\begin{eqnarray}}
\newcommand{\een}{\end{eqnarray}}
\newcommand{\beno}{\begin{eqnarray*}}
\newcommand{\eeno}{\end{eqnarray*}}

\let\D=\Delta

\def\dH{\dot{H}}
\def\T{\mathbb{T}}

\def\virgp{\raise 2pt\hbox{,}}
\def\cdotpv{\raise 2pt\hbox{;}}

\def\C{\mathop{\mathbb C\kern 0pt}\nolimits}
\def\DD{\mathop{\mathbb D\kern 0pt}\nolimits}
\def\EE{\mathop{\mathbb E\kern 0pt}\nolimits}
\def\K{\mathop{\mathbb K\kern 0pt}\nolimits}
\def\N{\mathop{\mathbb  N\kern 0pt}\nolimits}
\def\Q{\mathop{\mathbb  Q\kern 0pt}\nolimits}
\def\R{{\mathop{\mathbb R\kern 0pt}\nolimits}}
\def\SS{\mathop{\mathbb  S\kern 0pt}\nolimits}
\def\St{\mathop{\mathbb  S\kern 0pt}\nolimits}
\def\Z{\mathop{\mathbb  Z\kern 0pt}\nolimits}
\def\ZZ{{\mathop{\mathbb  Z\kern 0pt}\nolimits}}
\def\H{{\mathop{{\mathbb  H\kern 0pt}}\nolimits}}
\def\PP{\mathop{\mathbb P\kern 0pt}\nolimits}
\def\TT{\mathop{\mathbb T\kern 0pt}\nolimits}

\def\h {{\rm h}}
\def\v {{\rm v}}
\def\na{\nabla}
\def\pa{\partial}

\newcommand{\andf}{\quad\hbox{and}\quad}
\newcommand{\with}{\quad\hbox{with}\quad}

\def\dive{\mathop{\rm div}\nolimits}

\begin{document}

\title{DETERMINISTIC AND STOCHASTIC 2D NAVIER-STOKES EQUATIONS WITH ANISOTROPIC VISCOSITY}
\maketitle

\author{Siyu Liang \footnotemark[1],\footnotemark[2]
        Ping Zhang  \footnotemark[1],
        Rongchan Zhu \footnotemark[3],\footnotemark[4]}
        \footnotetext[1]{Academy of Mathematics and Systems Science, Chinese Academy of Sciences, Beijing 100190, China}
        \footnotetext[2]{School of Mathamatical Sciences, University of Chinese Academy of Sciences, Beijing 100049, China}
        \footnotetext[3]{Department of Mathematics, Beijing Institute of Technology, Beijing 100081, China}
        \footnotetext[4]{Department of Mathematics, University of Bielefeld, D-33615 Bielefeld, Germany}
\\

%

\begin{abstract}
In this paper, we investigate both deterministic and stochastic 2D Navier Stokes equations with anisotropic viscosity.
For the deterministic case, we prove the global well-posedness of the system with initial data  in the anisotropic
 Sobolev space $\tilde{H}^{0,1}.$  For the stochastic case,  we obtain existence of the martingale solutions
  and pathwise uniqueness of the solutions, which imply  existence of the probabilistically strong solution to
  this system by the Yamada-Watanabe Theorem.
\end{abstract}

\bibliographystyle{plain}

\section{Introduction}

We recall the incompressible classical Navier-Stokes system for incompressible fluids
\begin{equation*}
(NS_\nu)\quad \left\{
\begin{aligned}
& \partial_{t}u+u\cdot\nabla u-\nu \Delta u=-\nabla p, \\
&\dive u=0, \\
&u\mid_{t=0}=u_{0},
\end{aligned}
\right .
\end{equation*}
where $\nu>0$ is the viscosity of the fluid, $v$ and  $p$ denote the velocity and the  pressure of the fluid respectively.
In 1934, J. Leray  proved global existence of finite energy weak solutions to $(NS_\nu)$ in the whole space $\R^d,$ for $d= 2, 3,$ in the seminar paper \cite{leray}.
When $d=2,$  global weak solutions to $(NS_\nu)$ are unique. However, when $d=3,$ the regularities and uniqueness of Leray solutions to $(NS_\nu)$
 are still widely open in the field of mathematical fluid mechanics except the case when the norm of the initial data is small compared to the viscosity
$\nu.$

 Here we consider the incompressible Navier-Stokes equations with partial dissipation.
 Systems of this type appear in geophysical fluids( see for instance \cite{chemin2006mathematical, ped}).
 In fact, instead of putting the
classical viscosity $-\nu\D$ in $(NS_\nu),$ meteorologist often
modelize turbulent diffusion by putting a viscosity of the form:
$-\nu_h\D_h-\nu_3\pa_{x_3}^2,$ where $\nu_h$ and $\nu_3$ are
empiric constants, and $\nu_3$ is usually much smaller than
$\nu_h.$ We refer to the book of J. Pedlovsky \cite{ped}, Chapter~4 for a more complete discussion. And in the particular case of
the so-called Ekman layers  \cite{ekman1905influence, grenier1997ekman} for rotating fluids, $\nu_{3}=\epsilon \nu_{h}$ and $\epsilon$ is a very small parameter.
 In \cite{chemin2000fluids, chemin2007global,  paicu2005equation}, the authors consider the global well-posedness of such system with small initial
 data in some anisotropic Besov type spaces.
However, for this 3D anisotropic Navier-Stokes equation, there is no result concerning global existence of weak solutions. In this paper,
we concentrate on the 2D case first.

The aim of this paper is to investigate both the following deterministic system on $\mathbb{R}^{2}$ or on the two dimensional torus $\mathbb{T}^2$
\begin{equation}\label{S1eq1}
\left\{
\begin{aligned}
& \partial_{t}u+u\cdot\nabla u-\partial_{1}^{2}u=-\nabla p, \\
&\dive u=0, \\
&u\mid_{t=0}=u_{0},
\end{aligned}
\right .
\end{equation}
and the following stochastic  system on  $\mathbb{T}^2$
\begin{equation}\label{S1eq2}
\left\{
\begin{aligned}
& du+(u\cdot\nabla u-\partial _{1}^{2}u)dt=\sigma(t,u) dW-\nabla p dt, \\
&\dive u=0, \\
&u\mid_{t=0}=u_{0},
\end{aligned}
\right .
\end{equation}
where $\sigma$ is the random external force and $W$ is an $\ell^2$- cylindrical Wiener process, the definition of which will be introduced in Section 4.
It is well-known that for the 2D incompressible Euler system with initial data in $s$-order Sobolev space $H^{s}(\R^2)$ for $s>2$, there exists a unique global solution in $L^{\infty}_{\rm loc}(\R^+,H^{s}(\R^2))$ (see \cite{Bahouri2011Fourier} for instance). As \eqref{S1eq1} is an intermediate  equation between 2D Euler equations and 2D Navier Stokes equations, we also have similar global well-posedness for \eqref{S1eq1} with initial data  in $H^{s}(\R^2)$ for $s>2$.
 Since \eqref{S1eq1} is more dissipative than Euler system, we are going to prove its global well-posedness with initial data in $\tilde{H}^{0,1}$ (see Section 2 for the definition of $\tilde{H}^{0,1}$ ).  For the stochastic two dimensional Euler equation, we can only deal with it driven by additive or linear multiplicative noise (See \cite{Glatt2014Local}). But
  for the anisotropic system \eqref{S1eq1}, we can solve the martingale problem with general multiplicative noise. The main novelty is an $H^{0,1}$- uniform estimate, the proof of which depends  crucially on the divergence free condition
  (see \eqref{S3eq2}).

The plan of this paper is as follows. In Section \ref{Sect2}, we  introduce some notations and  recall some preliminaries. In the following two sections, we  first study the deterministic  equation \eqref{S1eq1} and then we consider \eqref{S1eq2} with stochastic external force which may depend on the velocity $u$.

\textbf{Main results for deterministic part}: We prove the existence and uniqueness of weak solutions in the space $L^{\infty}(\mathbb{R}^+; H^{0,1})\cap L^{2}(\mathbb{R}^+; \dH^{1,1})$ for the   deterministic equation \eqref{S1eq1} (see Theorem \ref{thm1} below). In order to prove existence results, we need both the  $L^{2}$  as well as $H^{0,1}$  uniform estimate
for appropriate approximate solutions to \eqref{S1eq1}. (The definition of $H^{0,1}$ and $\dH^{1,1}$ are given in Section 2)   Note that the only $L^{2}$ uniform estimate does not provide the compactness of the approximate
  solutions  due to the lack of the estimate for $\|\partial_{2}u\|_{L^{2}}.$  To obtain the uniform $H^{0,1}$ estimate,  we have to use the divergence free condition
  of the velocity field in the crucial  way. The uniqueness of  solutions is proved by   estimating the difference between any two solutions, $w=u-v,$ in the space $L^{2}$.

\textbf{Main results for martingale solution of stochastic equation}: We prove  the existence of martingale solutions (see Theorem 4.1 below). Similar 3D
 equations  with a Brickman-Forchheimer term, $|u|^{2\alpha}u,$ have been studied in \cite{bessaih2017stochastic}.
In this paper we prove  the existence and uniqueness of probabilistically strong solutions to (1.2) in dimension 2 without Brickman-Forchheimer term. In order to do so,
we first use Galerkin approximations to project \eqref{S1eq2} in finite dimensional space.  Then we use It\^{o}'s formula to obtain the uniform estimates of $u_{n}$ in both $L^{2}$ and $H^{0,1}$. Similar to the deterministic case, the proof depends  heavily on the divergence free condition of the velocity field. However, since we have to take the expectation, we can not use It\^{o}'s formula to estimate  $\|u_{n}\|_{H^{0,1}}^{2}$ directly, instead  we shall multiply it by an exponential term $e^{-2c\int_{0}^{t}\parallel \partial_{1}u\parallel_{L^{2}}^{2}dt}$ for some proper $c$. Then by tightness methods (Skorokhod Theorem), we can obtain the existence of martingale solutions. Here we emphasize that we rely more heavily on the divergence free condition and we could not use similar methods as in \cite{bessaih2017stochastic} since we do not have Brickman-Forchheimer term, (which helps to obtain a better estimate for the solution.) And we have to use the martingale approach. Moreover, we can prove the pathwise uniqueness of solutions in $L^{2}$ space. Finally by the Yamada-Watanabe theorem, we obtain the existence and uniqueness of the (probabilistically) strong solution to (1.2).

\setcounter{equation}{0}
\section{Preliminaries}\label{Sect2}
We first recall some function spaces on $\mathbb{R}^{2}$ and on the two dimensional torus  $\mathbb{T}^{2}$.

\subsection{Function spaces on $\mathbb{R}^{2}$}

On $\mathbb{R}^{2}$, we recall the classical Sobolev spaces:
$$H^{s}(\mathbb{R}^{2}):=\Bigl\{u\in \mathcal{S}'(\mathbb{R}^{2});\parallel u \parallel_{H^{s}(\mathbb{R}^{2})}^{2}:=\int_{\mathbb{R}^{2}}(1+\mid \xi \mid^{2})^s\mid \hat{u}(\xi)\mid^{2}d\xi<\infty\,\Bigr\}, $$
where $\hat{u}$ denotes the Fourier transform of $u$.
Due to the anisotropic properties of \eqref{S1eq1}, we also need
 anisotropic Sobolev spaces. Let us recall the anisotropic Sobolev norms and spaces:
$$H^{s,s'}(\mathbb{R}^{2}):=\Bigl\{u\in \mathcal{S}'(\mathbb{R}^{2});\parallel u \parallel_{H^{s,s'}(\mathbb{R}^{2})}^{2}:=\int_{\mathbb{R}^{2}}(1+\mid \xi_{1} \mid^{2})^s(1+ \mid \xi_{2} \mid^{2})^{s'}\mid \hat{u}(\xi)\mid^{2}d\xi<\infty\,\Bigr\},                                   $$
where $\xi=(\xi_{1},\xi_{2})$.\\
We remark that the space $H^{s,s'}(\mathbb{R}^{2})$ endowed with the norm $\parallel \cdot \parallel_{H^{s,s'}(\mathbb{R}^{2})}$ is a Hilbert space.
We also recall the horizontally homogeneous
 anisotropic Sobolev norm and the space:
$$\dH^{s,s'}(\mathbb{R}^{2}):=\Bigl\{u\in \mathcal{S}'(\mathbb{R}^{2});\| u \|_{\dH^{s,s'}(\mathbb{R}^{2})}^{2}:=\int_{\mathbb{R}^{2}}\mid \xi_{1} \mid^{2s}(1+ \mid \xi_{2} \mid^{2})^{s'}\mid \hat{u}(\xi)\mid^{2}d\xi<\infty\,\Bigr\}. $$

In what follows,  we  shall use `$\h$' to denote the horizonal variable $x_{1},$  and `$\v$' to  the vertical direction $x_2.$
Let $\mathbb{R}^{2}=(\mathbb{R}_{\h},\mathbb{R}_{\v}).$
 For exponents $p,q\in[1,\infty)$, we  denote the space $L^{p}(\mathbb{R}_{\h},L^{q}(\mathbb{R}_{\v}))$ by $L_{\h}^{p}(L_{\v}^{q})(\mathbb{R}^{2}),$
 which is endowed with the norm
 $$\parallel u \parallel_{L_{\h}^{p}(L_{\v}^{q})(\mathbb{R}^{2})}:=
 \Bigl\{\int_{\mathbb{R}_{\h}}\bigl(\int_{\mathbb{R}_{\v}}\mid u(x_{1},x_{2}) \mid^{q}dx_{2}\bigr)^{\frac{p}{q}}dx_{1}\Bigr\}^{\frac{1}{p}}.$$
 Similar notation for $L_{\v}^{p}(L_{\h}^{q})(\mathbb{R}^{2})$. Then it follows from Minkowski inequality that
 \beno
 \begin{split}
 \parallel u \parallel_{L_{\h}^{p}(L_{\v}^{q})(\mathbb{R}^{2})}\leq \parallel u \parallel_{L_{\v}^{q}(L_{\h}^{p})(\mathbb{R}^{2})}  \text{ when } 1\leq q\leq p\leq\infty,\\
 \parallel u\parallel_{L_{\v}^{q}(L_{\h}^{p})(\mathbb{R}^{2})} \leq \parallel u \parallel_{L_{\h}^{p}(L_{\v}^{q})(\mathbb{R}^{2})} \text{when } 1\leq p\leq q\leq\infty .
 \end{split}
 \eeno

\subsection{Function spaces on $\mathbb{T}^{2}$}
Now we recall some function spaces for the two dimensional torus $\mathbb{T}^2$.
Let $\mathbb{T}^{2}=\mathbb{R}/2\pi \mathbb{Z}\times\mathbb{R}/2\pi \mathbb{Z}=(\mathbb{T}_{\h},\mathbb{T}_{\v})$.
 Similar to the whole space $\mathbb{R}^{2}$, we recall the anisotropic $L^{p}$ spaces:
$$\parallel u \parallel_{L_{\h}^{p}(L_{\v}^{q})(\mathbb{T}^{2})}:=
 \Bigl\{\int_{\mathbb{T}_{\h}}\bigl(\int_{\mathbb{T}_{\v}}\mid u(x_{1},x_{2}) \mid^{q}dx_{2}\bigr)^{\frac{p}{q}}dx_{1}\Bigr\}^{\frac{1}{p}}.$$
 Similar to the whole space, we also have:
 \beno
 \begin{split}
 \parallel u \parallel_{L_{\h}^{p}(L_{\v}^{q})(\mathbb{T}^{2})}\leq& \parallel u \parallel_{L_{v}^{q}(L_{h}^{p})(\mathbb{T}^{2})}  \text{ when } 1\leq q\leq p\leq\infty, \\
 \parallel u\parallel_{L_{\v}^{q}(L_{\h}^{p})(\mathbb{T}^{2})} \leq& \parallel u \parallel_{L_{h}^{p}(L_{\v}^{q})(\mathbb{T}^{2})}  \text{ when } 1\leq p\leq q\leq\infty.
 \end{split}
 \eeno

For $u\in L^{2}(\T^2)$, we consider the Fourier expansion of $u$:
 $$u(x)=\sum\limits_{k\in \mathbb{Z}^{2} }\hat{u}_{k}e^{ ik\cdot x} \with
 \hat{u}_{k}=\overline{\hat{u}_{-k}},$$
 where $\hat{u}_{k}:=\frac{1}{2\pi}\int_{[0,2\pi]\times [0,2\pi]}u(x)e^{- ik\cdot x}dx$ denotes the Fourier coefficient of $u$ on $\mathbb{T}^{2}$. It follows from Fourier-Plancherel equality that the series is convergent  in $L^{2}(\mathbb{T}^{2})$.

 Define the Sobolev norm :
 $$\parallel u \parallel_{H^{s}(\mathbb{T}^{2})}^{2}:=\sum\limits_{k\in \mathbb{Z}^{2} }(1+\mid k \mid^{2} )^s\mid \hat{u}_{k}\mid^{2},$$
 and the anisotropic Sobolev norms:
  \beno
  \begin{split}
  &\parallel u \parallel_{H^{s,s'}(\mathbb{T}^{2})}^{2}=\sum\limits_{k\in \mathbb{Z}^{2} }(1+\mid k_{1} \mid^{2} )^s(1+\mid k_{2} \mid^{2} )^{s'}\mid \hat{u}_{k}\mid^{2},\\
  &\parallel u \parallel_{\dH^{s,s'}(\mathbb{T}^{2})}^{2}=\sum\limits_{k\in \mathbb{Z}^{2} }\mid k_{1} \mid^{2s}(1+\mid k_{2} \mid^{2} )^{s'}\mid \hat{u}_{k}\mid^{2},
  \end{split}
  \eeno
  where $k=(k_{1},k_{2})$.\\
And we also define the Sobolev spaces $H^{s}(\mathbb{T}^{2})$, $H^{s,s'}(\mathbb{T}^{2})$ and $\dH^{s,s'}(\mathbb{T}^{2})$ as the completion of $C^{\infty}(\mathbb{T}^{2})$ with the norms $\parallel \cdot \parallel_{H^{s}(\mathbb{T}^{2})}$, $\parallel \cdot \parallel_{H^{s,s'}(\mathbb{T}^{2})}$ and $\parallel \cdot \parallel_{\dH^{s,s'}(\mathbb{T}^{2})}$ respectively.

 \subsection{Some other notations and definitions}
  We use $D$ to denote the domain $\mathbb{R}^{2}$ or $\mathbb{T}^{2}$.
Let us denote
\beno
\begin{split}
H:=&\left\{u\in L^{2}(D);\ \dive u =0 \ \right\},\\
V:=&\left\{u\in H^{1}(D),\ \dive u =0\ \right\},\\
\tilde{H}^{s,s'}:=&\left\{u\in H^{s,s'}(D),\ \dive u =0\ \right\}.
\end{split}
\eeno
Moreover, we use $(\cdot,\cdot)$ or $(\cdot \mid \cdot)$ to denote the scalar product
$$(u,v)=(u \mid v)=(u,v)_{L^{2}(D)}=\sum\limits_{j=1}^{2}\int_{D}u_{j}(x)v_{j}(x)dx.$$
We use $(\cdot,\cdot)_{H^{0,1}}$ or $(\cdot,\cdot)_{0,1}$  to denote the inner product
 $$(u,v)_{H^{0,1}(D)}=\sum\limits_{j=1}^{2}\int_{D}\bigl(u_{j}(x)v_{j}(x)+\partial_{2}u_{j}(x)\partial_{2}v_{j}(x)\bigr)\,dx.$$
 Let $$\textbf{P}:L^{2}(D)\rightarrow H  \text{ is the Leray projection operator to divergence free space}.$$
 By applying $\textbf{P}$ to \eqref{S1eq1}, we write
 $$\partial_{t}u=\textbf{P}(\partial_{1}^{2}u-u\cdot \nabla u). $$

As usual, when $u,v,w\in H^{1}(D)$, we denote
\beno
\begin{split}
B(u,v):=&u\cdot \nabla v,\\
B(u):=&u\cdot \nabla u,\\
b(u,v,w):=&(u\cdot \nabla v, w).
\end{split} \eeno
Then we have
 $b(u,v,w)=-b(u,w,v)$ for $u,v,w \in V$.
In particular, $b(u,v,v)=0$.

Let us end this section by  the definition of weak solution to (1.1)

\begin{defi}[ weak solution]\label{S2def1}
We call $u$ a global weak solution of (1.1) with the initial data $u_{0}$ if $u$ satisfies:
\begin{itemize}
\item[(i)] $u\in L^{\infty}(\R^+; \tilde{H}^{0,1}(D))\cap L^{2}(\R^+; \dH^{1,1}(D))$ ;
\item[(ii)] for any $\varphi\in C_{0}^{\infty}(D)$ with  $\dive\varphi=0$, and $t>0,$
\begin{equation}\label{S2eq1}
\int_{0}^{t}\bigl\{-(u, \partial_{t}\varphi)+(\partial_{1} u, \partial_{1}\varphi)+( u\cdot\nabla u, \varphi)\bigr\}\,ds
=(u_{0}, \varphi(0))-( u(t), \varphi(t)),
\end{equation}
\end{itemize}
  \end{defi}

\setcounter{equation}{0}
\section{The Deterministic Case}\label{Sect3}

 For simplicity, we always omit the domain $D$ in this section.

\begin{thm}\label{thm1}
{\sl Given solenoidal vector field  $u_{0}$ in $\tilde{H}^{0,1},$  \eqref{S1eq1} has a unique global weak solution $u\in
L^{\infty}(\R^+; \tilde{H}^{0,1})\cap L^{2}(\R^+; \dH^{1,1})$ in the sense of Definition \ref{S2def1}.}
\end{thm}

\begin{lem}\label{S3lem1}
{\sl Let $u$ be a global smooth enough solution to \eqref{S1eq1}. Then one has
\begin{equation}\label{S3eq1}
\| u(t) \|_{L^{2}}^{2}+2\int_{0}^{t}\| \partial_{1}u(s) \|_{L^{2}}^{2}ds\leq \| u_0\| _{L^{2}}^{2}.
\end{equation}}
\end{lem}
\begin{proof}
Indeed by taking the ${L^{2}}$ inner product of the momentum equation of \eqref{S1eq1} with $u$ and using $\dive u=0,$ we obtain
$$\frac{1}{2}\frac{d}{dt}\| u(t) \| _{L^{2}}^{2}+\| \partial_{1}u \|_{L^{2}}^{2}=0.$$
Integrating the above inequality over $[0,t]$ leads to \eqref{S3eq1}.
\end{proof}

\begin{lem}\label{S3lem2}
{\sl Under the same assumption of Lemma \ref{S3lem1}, we have
\begin{equation}\label{S3eq2}
\begin{split}
\|\partial_{2}u(t)\|_{L^{2}}^{2}+\int_{0}^{t}\|\partial_{1}\partial_{2}u(s)\|_{L^{2}}^{2}ds
&\leq
\|\partial_{2}u_0\|_{L^{2}}^{2}e^{C\| u_0\|_{L^{2}}^{2}}
\end{split}
\end{equation}
for some constant $C>0$.}
\end{lem}
\begin{proof}
By Taking $\pa_2$ to the momentum equation of \eqref{S1eq1} and   then taking $L^2$ inner product of the resulting equation with $\partial_{2}u$, we obtain
\beq \label{S3eq3}
\frac{1}{2}\frac{d}{dt}\| \partial_{2}u(t) \| _{L^{2}}^{2}+
\| \partial_{1}\partial_{2}u \|_{L^{2}}^{2}\leq -(\partial_{2}(u\cdot \nabla u )\mid \partial_{2}u).\eeq
It is easy to observe that
\begin{equation}\label{S3eq4}
(\partial_{2}(u\cdot \nabla u )\mid \partial_{2}u)=(\partial_{2}(u\cdot \nabla u^{1} )\mid \partial_{2}u^{1})+(\partial_{2}(u\cdot \nabla u^{2} )\mid \partial_{2}u^{2}),
\end{equation}
where $u=(u^1,u^2)$.\\
For the first term on the right-hand side of \eqref{S3eq4}, we have
\begin{equation*}
\begin{split}
(\partial_{2}(u\cdot \nabla u^{1} )\mid \partial_{2}u^{1})=&(\partial_{2}(u^{1}\partial_{1}u^{1}+u^{2}\partial_{2}u^{1})\mid \partial_{2}u^{1})\\
=&(\partial_{2}u^{1}\partial_{1}u^{1}\mid \partial_{2}u^{1})+(u^{1}\partial_{2}\partial_{1}u^{1}\mid \partial_{2}u^{1})\\
&+
(\partial_{2}u^{2}\partial_{2}u^{1}\mid \partial_{2}u^{1})+(u^{2}\partial_{2}^{2}u^{1}\mid \partial_{2}u^{1}),
\end{split}
\end{equation*}
Yet due to $\dive u=0,$ we achieve
\beno
(\partial_{2}u^{1}\partial_{1}u^{1}\mid \partial_{2}u^{1})+
(\partial_{2}u^{2}\partial_{2}u^{1}\mid \partial_{2}u^{1})=0,
\eeno
and
\begin{equation*}
\begin{split}
(u^{1}\partial_{2}\partial_{1}u^{1}\mid \partial_{2}u^{1})+&(u^{2}\partial_{2}^{2}u^{1}\mid \partial_{2}u^{1})
=\bigl(u\cdot\na\pa_2u^1 |\pa_2u^1\bigr)\\
&=\frac12\int_{D}u\cdot\na |\pa_2u^1|^2\,dx=-\frac12\int_{D}\dive u |\pa_2u^1|^2\,dx=0.
\end{split}
\end{equation*}
This leads to
\beq \label{S3eq5}
(\partial_{2}(u\cdot \nabla u^{1} )\mid \partial_{2}u^{1})=0.
\eeq

For the second term on the right-hand side of \eqref{S3eq4}, again due to $\dive u=0$, we have:
\begin{equation*}
\begin{split}
(\partial_{2}(u\cdot \nabla u^{2} )\mid \partial_{2}u^{2})=&(\partial_{2}(u^{1}\partial_{1}u^{2})\mid \partial_{2}u^{2})+(\partial_{2}(u^{2}\partial_{2}u^{2})\mid \partial_{2}u^{2})\\
=&\bigl(\pa_2u\cdot\na u^2 | \pa_2u^2\bigr).
\end{split}
\end{equation*}
The second equality is due to
$$(u^{1}\partial_{1}\partial_{2}u^{2}\mid \partial_{2}u^{2})+(u^{2}\partial_{2}^{2}u^{2}\mid \partial_{2}u^{2})
=-\frac{1}{2}(\partial_{1}u^{1}\partial_{2}u^{2}\mid \partial_{2}u^{2})-\frac{1}{2}(\partial_{2}u^{2}\partial_{2}u^{2}\mid \partial_{2}u^{2})=0.$$
Whereas notice that
\begin{equation*}
\begin{split}
&\mid\bigl(\pa_2u\cdot\na u^2 | \pa_2u^2\bigr)\mid\\
&=\mid(\partial_{2}u^{1}\partial_{1}u^{2}\mid \partial_{2}u^{2})+(\partial_{2}u^{2}\partial_{2}u^{2}\mid \partial_{2}u^{2})\mid\\
&\leq \bigl(\| \partial_{2}u^{1}\|_{L_{\h}^{\infty}(L_{\v}^{2})}\| \partial_{1}u^{2}\| _{L_{\h}^{2}(L_{\v}^{\infty})}+\| \partial_{1}u^{1}\| _{L_{\h}^{2}(L_{\v}^{\infty})}
\| \partial_{2}u^{2}\|_{L_{\h}^{\infty}(L_{\v}^{2})}\bigr)\| \partial_{2}u^{2}\|_{L^{2}},
\end{split}
\end{equation*}
from which and
\beq \label{S3eq6}
\| u \|_{L^2_\v(L^\infty_\h)}\lesssim \| u\|_{L^{2}}^{\frac{1}{2}}\| \partial_{x_1}u\|_{L^{2}}^{\frac{1}{2}}\andf
\| u \|_{L^2_\h(L^\infty_\v)}\lesssim \| u\|_{L^{2}}^{\frac{1}{2}}\| \partial_{x_2}u\|_{L^{2}}^{\frac{1}{2}} ,\eeq
we infer
\begin{equation*}
\begin{split}
\mid\bigl(\pa_2u\cdot\na u^2 | \pa_2u^2\bigr)\mid
\lesssim\bigl(&\| \partial_{1}\partial_{2}u^{1}\|_{L^{2}}^{\frac{1}{2}}
 \| \partial_{2}u^{1}\|_{L^{2}}^{\frac{1}{2}}\| \partial_{1}u^{2}\|_{L^{2}}^{\frac{1}{2}}\| \partial_{1}\partial_{2}u^{2}\|_{L^{2}}^{\frac{1}{2}}\\
 &+\| \partial_{1}\partial_{2}u^{1}\|_{L^{2}}^{\frac{1}{2}}\| \partial_{1}u^{1}\|_{L^{2}}^{\frac{1}{2}}\| \partial_{1}\partial_{2}u^{2}\|_{L^{2}}^{\frac{1}{2}}\| \partial_{2}u^{2}\|_{L^{2}}^{\frac{1}{2}}\bigr)\| \partial_{2}u^{2}\|_{L^{2}}.
 \end{split}
\end{equation*}
 This together with $\dive u=0$ ensures that
 \beno
 \mid\bigl(\pa_2u\cdot\na u^2 | \pa_2u^2\bigr)\mid
\lesssim \| \partial_{1}\partial_{2}u\|_{L^{2}}\| \partial_{1}u\|_{L^{2}}\| \partial_{2}u\|_{L^{2}}.
\eeno
Along with \eqref{S3eq5}, we achieve
\begin{equation*}
\mid(\partial_{2}(u\cdot \nabla u )\mid \partial_{2}u)\mid \lesssim \| \partial_{1}\partial_{2}u\|_{L^{2}}\| \partial_{1}u\|_{L^{2}}\| \partial_{2}u\|_{L^{2}}
\end{equation*}
 Applying Young's inequality yields
$$\mid(\partial_{2}(u\cdot \nabla u )\mid \partial_{2}u)\mid\leq \frac{1}{2}\|\partial_{1}\partial_{2}u\|_{L^{2}}^{2}+C\| \partial_{1}u\|_{L^{2}}^{2}\| \partial_{2}u\|_{L^{2}}^{2}.$$
Inserting the above inequality into \eqref{S3eq3} gives rise to
$$\frac{d}{dt}\| \partial_{2}u(t) \| _{L^{2}}^{2}
+\| \partial_{1}\partial_{2}u \|_{L^{2}}^{2}\leq
2C\| \partial_{1}u\|_{L^{2}}^{2}\| \partial_{2}u\|_{L^{2}}^{2}.$$
Applying Gronwall's inequality and using \eqref{S3eq1}, we obtain
\begin{equation*}
\begin{split}
\| \partial_{2}u(t) \| _{L^{2}}^{2}+\int_{0}^{t}\| \partial_{1}\partial_{2}u(t) \|_{L^{2}}^{2}ds&\leq
e^{2C\int_{0}^{t}\| \partial_{1}u\|_{L^{2}}^{2}ds}\| \partial_{2}u_0 \| _{L^{2}}^{2}\\
&\leq
e^{2C\| u_0\|_{L^{2}}^{2}}\| \partial_{2}u_0 \| _{L^{2}}^{2},
\end{split}
\end{equation*}
which yields \eqref{S3eq2}.

It remains to prove \eqref{S3eq6}. We only present the proof to the first one, the second one follows along the
same line. Indeed observing that
\beno
\begin{split}
f^2(x_1,x_2)=\int_{-\infty}^{x_1}\pa_y f^2(y,x_2)\,dy=&2\int_{-\infty}^{x_1}f(y,x_2)\pa_y f(y,x_2)\,dy\\
\leq& 2\|f(\cdot,x_2)\|_{L^2_\h}\|\pa_{x_1}f(\cdot,x_2)\|_{L^2_\h},
\end{split}
\eeno
which implies
\beno
\|f(\cdot,x_2)\|_{L^\infty_\h}^2\leq 2\|f(\cdot,x_2)\|_{L^2_\h}\|\pa_{x_1}f(\cdot,x_2)\|_{L^2_\h}.
\eeno
Applying H\"older's inequality in the $x_2$ variable gives
\beno
\|f\|_{L^2_\v(L^\infty_\h)}^2\leq 2\|f\|_{L^2}\|\pa_{x_1}f\|_{L^2}.
\eeno
This completes the proof of the lemma.
\end{proof}

Let us now present the proof of Theorem \ref{thm1}.

\begin{proof}[Proof of Theorem \ref{thm1}] We divide the proof of this theorem to the following two parts:

(1) Existence part. It is standard that the first step to prove the existence of weak solutions to some nonlinear partial
differential equations is to construct appropriate approximate solutions. Here we
consider
\begin{equation}\label{S3eq7}
\left\{
\begin{aligned}
& \partial_{t}u^{\epsilon}+u^{\epsilon}\cdot\nabla u^{\epsilon}-\partial_{1}^{2}u^{\epsilon}-\epsilon^{2}\partial_{2}^{2}u^{\epsilon}
=-\nabla p^{\epsilon} \\
&\dive u^{\epsilon}=0 \\
&u^{\epsilon}(0)=u_{0}\ast j_{\epsilon}
\end{aligned}
\right .,
\end{equation}
 where $j$ is a smooth function on $\mathbb{R}^{2}$ with
 $$j(x)=1,\text{  } \mid x \mid\leq 1 ;  \text{     }j(x)=0,\text{  } \mid x \mid\geq 2,$$
 and
 $$j_{\epsilon}(x)=\frac{1}{\epsilon^{2}}j(\frac{x}{\epsilon}).$$
It follows from classical theory on Navier-Stokes system that \eqref{S3eq7} has a unique global smooth solution $(u^\epsilon, p^\epsilon)$ for any fixed $\epsilon.$
Furthermore, along the same line to the proof of Lemmas \ref{S3lem1} and \ref{S3lem2}, we have
\beq \label{S3eq8}
\begin{split}
&\| u^\epsilon(t) \|_{L^{2}}^{2}+2\int_{0}^t\| \partial_{1}u^\epsilon \|_{L^{2}}^{2}ds+\epsilon^2\int_{0}^t\|\pa_2u^\epsilon \|_{L^{2}}^2ds\leq \| u_0\| _{L^{2}}^{2},\\
&\|\partial_{2}u^\epsilon(t)\|_{L^{2}}^{2}+\int_{0}^t\|\partial_{1}\partial_{2}u^\epsilon\|_{L^{2}}^{2}ds+\epsilon^2\int_{0}^t\|\partial_{2}^2u^\epsilon\|_{L^{2}}^{2}ds
\leq
\|\partial_{2}u_0\|_{L^{2}}^{2}e^{C\| u_0\|_{L^{2}}^{2}}.
\end{split}
\eeq
It is obvious that for $\varphi\in C_{0}^{\infty}$ with $\dive\varphi=0$, $u^{\epsilon}$ satisfies the following equation:
\begin{equation}\label{S3eq9}
\int_{0}^{t}\bigl(-( u^{\epsilon},\partial_{t}\varphi )+ ( u^{\epsilon}\cdot\nabla u^{\epsilon}, \varphi)+(\partial_{1}u^{\epsilon},\partial_{1}\varphi)
+\epsilon^{2}(\partial_{2}u^{\epsilon},\partial_{2}\varphi)\bigr)\,ds
=0
\end{equation}
Then for any fixed $T>0$,
$\bigl\{u^{\epsilon}\bigr\}_{\epsilon>0}$ is uniformly bounded in $L^{\infty}([0,T]; H^{0,1})\cap L^{2}([0,T]; H^{1,1})$.
By  interpolation, $\bigl\{u^{\epsilon}\bigr\}_{\epsilon>0}$ is uniformly bounded in $L^{4}([0,T];H^{\frac{1}{2}})$. Sobolev imbedding implies that
 $\bigl\{u^{\epsilon}\bigr\}_{\epsilon>0}$ is bounded in $L^{4}([0,T];L^{4})$. Hence the nonlinear term in \eqref{S3eq7} is bounded in $L^{2}([0,T];H^{-1})$. Moreover, $\nabla p^\epsilon=\nabla \Delta^{-1}\partial_{i}\partial_{j}\left((u^{\epsilon})^i(u^\epsilon)^{j}\right)$ is uniformly bounded in $L^{2}([0,T];H^{-1})$. As
 a result, it comes out that
     \beq \label{S3eq10}
     \bigl\{\pa_t u^{\epsilon}\bigr\}_{\epsilon>0}\quad\mbox{ is  uniformly bounded in}\  L^{2}([0,T];H^{-1}).
     \eeq

At this stage, we need to use the following Aubin-Lions lemma:

\begin{lem}[Aubin-Lions ]
{\sl Let $K$ be the torus or a smooth bounded domain. If the sequence $\left(u_{n}\right)_{n\in\N}$ is uniformly
 bounded sequence in $L^{q}([0,T];H^{1}(K))$ for $q\in (1,\infty)$,
and $\left(\partial_{t}u_{n}\right)_{n\in\N}$ is a uniformly bounded sequence in $L^{p}([0,T];H^{1}(K))$ for some $p\in (1,\infty)$,
then there exist $u\in L^{q}([0,T]; H^1(K))$ and a subsequence of $\left(u_{n_j}\right)_{j\in\N}$ so that  $\left(u_{n_j}\right)_{j\in\N}$
 converges strongly to $u$ in $L^{q}([0,T];L^{2}(K))$.}
\end{lem}

Let us now take $\epsilon=\frac1n$ in \eqref{S3eq7}. Set $u_{n}=u^{\frac{1}{n}}$.\\
 (i) For torus $\mathbb{T}^{2}$ case,  given any $T>0,$ it follows from \eqref{S3eq9}, \eqref{S3eq10} and  Aubin-Lions Lemma that there is a subsequence, which
 we still denote by $\bigl\{u_n\bigr\}_{n\in\N}$ and some $u\in L^{\infty}([0,T]; H^{0,1})\cap L^{2}([0,T]; H^{1})$  so that
 \beq \label{S3eq11}
 u_{n}\rightarrow u \text{ } \text{strongly in}\text{ } L^{2}([0,T];L^{2}(\mathbb{T}^{2})) \text{ as }
\text{ }n \rightarrow \infty.  \eeq
Through a diagonal process with respect to $T,$ we can choose a subsequence, $\bigl\{u_n\bigr\}_{n\in\N}, $ so that \eqref{S3eq11} holds for any $T>0.$
Then we can pass the limit in \eqref{S3eq9} to obtain (\ref{S2eq1}).

(ii) For the case that $D=\mathbb{R}^{2}$, we choose a sequence of compact sets $(K_{i})$, such that $K_{i}\subset K_{i+1}$, and $\bigcup\limits_{i=1}^{\infty}K_{i}=\mathbb{R}^{2}$. By a classical diagonal methods, we can choose a subsequence of $\left(u_{n}\right)_{n\in\N}$
(which we still denote by  $\left(u_{n}\right)_{n\in\N}$ for simplicity) so that
$$u_{n}\rightarrow u \text{ } \text{strongly in}\text{ } L^{2}([0,T];L^{2}(K_{i}))
\text{ for any } i .$$
Since the test function $\varphi$ in (2.1) satisfies $\varphi\in C_{0}^{\infty}(D)$, it must be supported in some $K_{i}$.
Then as in case (i), we can take $n\to\infty$ in \eqref{S3eq9} to  obtain (\ref{S2eq1}).\\
Finally notice that since $u^{\epsilon}$ is uniformly bounded in $L^{\infty}(\R^+; H^{0,1})\cap L^{2}(\R^+; \dH^{1,1})$,  we can choose a subsequence of $u_{n}$ (which we denote by
$u_{n}$ again) and some $\tilde{u}$, such that \\
$u_{n}\rightarrow \tilde{u}$  weakly in $L^{2}([0,T]; H^{1,1})$ for each $T>0$, and \\
$u_{n}\rightarrow \tilde{u}$ weakly star in $L^{\infty}([0,T]; H^{0,1})$ for each $T>0$.\\
By the uniqueness of the limits of weak convergence, $u$ and $\tilde{u}$ coincide. \\
Since $u_{n}$ is uniformly bounded in $L^{\infty}(\R^+; H^{0,1})\cap L^{2}(\R^+; \dH^{1,1})$,
we have
$$\parallel u \parallel_{L^{\infty}(\R^+; H^{0,1})}=\lim\limits_{T\rightarrow \infty}\parallel u \parallel_{L^{\infty}([0,T]; H^{0,1})}\leq \sup\limits_{n}\parallel u_{n}\parallel_{L^{\infty}(\R^+; H^{0,1})},$$
and
$$\parallel u \parallel_{L^{2}(\R^+; \dH^{1,1})}=\lim\limits_{T\rightarrow \infty}\parallel u \parallel_{L^{2}([0,T]; \dH^{1,1})}\leq \sup\limits_{n}\parallel u_{n}\parallel_{L^{2}(\R^+; \dH^{1,1})}.$$

Thus we actually have $u\in L^{\infty}(\R^+; H^{0,1})\cap L^{2}(\R^+; \dH^{1,1})$.

(2) Uniqueness part.
Let  $u,v\in  L^{\infty}(\R^+, H^{0,1})\cap L^{2}(\R^+,\dH^{1,1})$ be two  weak solutions of \eqref{S1eq1}. We denote
 $w:=u-v$.
Then we have
$$\partial_{t}w+w\cdot \nabla v+u\cdot \nabla w-\partial_{1}^{2}w=-\nabla p.$$
Taking  $L^{2}$ inner product of the above equation with $w$ gives
\beq \label{S3eq12}
\frac{1}{2}\frac{d}{dt}\| w(t) \| _{L^{2}}^{2}+\| \partial_{1}w \|_{L^{2}}^{2}\leq \mid(w\cdot \nabla v\mid w)\mid.  \eeq
Observing that
\begin{equation}\label{S3eq30}
\begin{split}
\mid(w\cdot \nabla v\mid w)\mid
=&\mid(w^{1}\partial_{1}v+w^{2}\partial_{2}v\mid w)\mid\\
\leq &\bigl(\| w^{1}\|_{L_{\h}^{\infty}(L_{\v}^{2})}\|\partial_{1}v\|_{L_{\h}^{2}(L_{\v}^{\infty})}
  + \| w^{2}\|_{L_{\h}^{2}(L_{\v}^{\infty})} \|\partial_{2}v\|_{L_{\h}^{\infty}(L_{\v}^{2})}\bigr)\| w \|_{L^{2}},
  \end{split}
\end{equation}
where $w=(w^{1},w^{2})$,
from which and \eqref{S3eq6}, we deduce that
\begin{equation}\label{S3eq31}
\begin{split}
\mid(w\cdot \nabla v\mid w)\mid
 \lesssim& \bigl( \| w \|_{L^{2}}^{\frac{1}{2}}\| \partial_{1}w \|_{L^{2}}^{\frac{1}{2}}
\| \partial_{1}v \|_{L^{2}}^{\frac{1}{2}}\| \partial_{1}\partial_{2}v \|_{L^{2}}^{\frac{1}{2}}\\
&+\| w \|_{L^{2}}^{\frac{1}{2}}\| \partial_{1}w \|_{L^{2}}^{\frac{1}{2}}
\| \partial_{2}v \|_{L^{2}}^{\frac{1}{2}}\| \partial_{1}\partial_{2}v \|_{L^{2}}^{\frac{1}{2}}\bigr)\| w \|_{L^{2}}.
\end{split}
\end{equation}
Applying Young's inequality and using the divergence free condition $\partial_{2}w^{2}=-\partial_{1}w^{1}$ we have
\begin{equation*}
\begin{split}
\mid(w\cdot \nabla v\mid w)\mid
\leq \frac{1}{2}\| \partial_{1}w \|_{L^{2}}^{2}+C_{0}\bigl(&\| \partial_{1}v \|_{L^{2}}^{\frac{2}{3}}\| \partial_{1}\partial_{2}v \|_{L^{2}}^{\frac{2}{3}}
+\| \partial_{2}v \|_{L^{2}}^{\frac{2}{3}}\| \partial_{1}\partial_{2}v \|_{L^{2}}^{\frac{2}{3}}\bigr)
\| w \|_{L^{2}}^{2}.
\end{split}
\end{equation*}
Inserting the above inequality into \eqref{S3eq12} and applying
Gronwall's inequality we obtain
\begin{equation*}
\| w(t) \|_{L^{2}}^{2}\leq \| w_0 \| _{L^{2}}^{2}e^{2C_{0}\int_{0}^{t}\bigl(\| \partial_{1}v \|_{L^{2}}^{\frac{2}{3}}\| \partial_{1}\partial_{2}v \|_{L^{2}}^{\frac{2}{3}}
+\| \partial_{2}v \|_{L^{2}}^{\frac{2}{3}}\| \partial_{1}\partial_{2}v \|_{L^{2}}^{\frac{2}{3}}\bigr)\,ds}.
\end{equation*}
This along with the fact that $\| \partial_{1}v \|_{L^{2}}^{\frac{2}{3}}\| \partial_{1}\partial_{2}v \|_{L^{2}}^{\frac{2}{3}}
+\| \partial_{2}v \|_{L^{2}}^{\frac{2}{3}}\| \partial_{1}\partial_{2}v \|_{L^{2}}^{\frac{2}{3}}
$ belongs to $L^{1}([0,T])$ ensures
  $w(t)=0$, that is
 $u=v$. This completes the uniqueness part of the theorem.
\end{proof}

\setcounter{equation}{0}
\section{The Stochastic Case}\label{Sect4}
For the stochastic case, we consider the equation (1.2) on $\mathbb{T}^{2}$,
 and  again for simplicity of the notation, we always omit the domain $\mathbb{T}^{2}$ in what follows.
\subsection{Prelimaries and notations}
Let $(e_{k},k \geq1)$ be an orthonormal basis of $H$ whose elements belong to $H^{2}$ and orthogonal in $\tilde{H}^{0,1}$. For integers $k,l\geq 1$ with $k\neq l$, we deduce that
$$(\partial_{2}^{2}e_{k},e_{l})=-(\partial_{2}e_{k},\partial_{2}e_{l})=0.$$
Therefore, $\partial_{2}^{2}e_{k}$ is a constant multiple of $e_{k}$.

Let $\mathcal{H}_{n}=span(e_{1},,..,e_{n})$ and let $P_{n}$(resp. $\tilde{P}_{n}$ ) denote the orthogonal projection from $H$ (resp. $\tilde{H}^{0,1}$) to $\mathcal{H}_{n}$. We deduce that
$$ \text{   }P_{n}u=\tilde{P}_{n}u ,\text{   } \text{for}\text{   } u\in \tilde{H}^{0,1} .$$
Indeed, for $v\in \mathcal{H}_{n}$, we have
$\partial_{2}^{2}v\in \mathcal{H}_{n}$ and for any $u\in \tilde{H}^{0,1}$ :
$$(P_{n}u, v)=(u,v) \text{ and }\text{   } (\partial_{2}P_{n}u, \partial_{2}v)=-(P_{n}u,\partial_{2}^{2}v)=-(u,\partial_{2}^{2}v)=(\partial_{2}u,\partial_{2}v).
$$
Hence given $u\in \tilde{H}^{0,1}$, we have
$$(P_{n}u, v)_{0,1}=(u, v)_{0,1}, \text{   }\text{for any} \text{   } v\in \mathcal{H}_{n}.$$
This proves that $P_{n}$ and $\tilde{P}_{n}$ coincide on $\tilde{H}^{0,1}$.

Let $(W(t),t \geq 0)$ be an $\ell^2$-cylindrical Wiener process on a stochastic basis $(\Omega,\mathcal{F},P)$. Let $W_{n}(t)=\sum \limits_{j=1}^{n}\psi_{j}\beta_{j}(t):=\Pi_{n} W(t)$, where $\{\beta_{j}(t)\}$ is a sequence of independent Brownian Motions on $(\Omega, \mathcal{F}, P)$ and $\psi_{j}$ is an orthonormal basis of $\ell^2$.

Let $L^{2}(\ell^2,\mathbb{U})$ denotes the  Hilbert-Schmidt norms from $\ell^2$ to $\mathbb{U}$ for Hilbert space $\mathbb{U}$. For a Polish space $\mathbb{V}$, let $\mathcal{B}(\mathbb{V})$ denote its Borel $\sigma$-algebra and $\mathcal{P}(\mathbb{V})$ denote all the probability measures on $(\mathbb{V},\mathcal{B}(\mathbb{V}))$.
 Let $\sigma$ be a measurable mapping from $\bigl([0,T]\times \tilde{H}^{1,1},\mathcal{B}([0,T]\times \tilde{H}^{1,1})\bigl)$ to $\bigl(L^2(\ell^2,\tilde{H}^{1,1}),\mathcal{B}(L^{2}(\ell^2,\tilde{H}^{1,1}))\bigl)$.
 Then we introduce probabilistically weak, strong solutions and martingale solutions.
Set
$$F(u):=-B(u)+\partial_{1}^{2}u.$$

\begin{defi}[(Probabilistically) weak solution]
{\sl We say that a pair $(u,W)$ is a (probabilistically) weak solution to (1.2) if there exists a stochastic basis $(\Omega,\mathcal{F},\mathcal{F}_{t},P)$ such that  $u=(u(t))_{t\geq 0} $ is an $(\mathcal{F}_{t})$ adapted process and $W$ is an $\ell^2$-cylindrical Wiener process on $(\Omega,\mathcal{F},\mathcal{F}_{t},P)$ and the following holds:
\begin{itemize}
\item[(i)] $u\in L^{\infty}([0,T],\tilde{H}^{0,1})\cap L^{2}([0,T],\tilde{H}^{1,1}) $ for $a.s.$ $P$ and any $T>0$ ;

\item[(ii)] $\int_{0}^{T}\| F(u(s))\|_{H^{-1}}ds+\int_{0}^{T}\| \sigma(s,u(s))\|_{L^2(\ell^2,H)}^2ds<+\infty$  $a.s.$ $P$, for any $T>0$;

\item[(iii)] For every $l\in C^{1}(\mathbb{T}^{2})$ with  $\dive l =0$, $a.s.P$
$$u(0)=u_0,$$
$$\langle u(t),l\rangle =\langle u_0,l\rangle +\int_{0}^{t}\langle-u\cdot\nabla u+\partial_{1}^{2}u,l\rangle\,ds+\int_{0}^{t}\langle \sigma(s,u(s))dW(s),l\rangle.$$
\end{itemize}
}
\end{defi}
Here $\langle \cdot , \cdot\rangle$ denotes the duality bracket.
$\langle u , v \rangle$ and $(u,v)$ coincide when $u,v\in L^2$.

Now we define the (probabilistically) strong solution of (1.2) and
we fix a stochastic basis $(\Omega, \mathcal{F},P)$ and an $\ell^2$-cylindrical Wiener process $W$ on it.

\begin{defi}[(Probabilistically) strong solution]
{\sl We say that $u$ is a (probabilistically) strong solution to the equation (1.2) on the given probability space $(\Omega, \mathcal{F},P)$ with respect to the fixed cylindrical Wiener process $W$, if it satisfies:
\begin{itemize}
\item[(i)] $u$ is adapted to the filtration $\hat{\mathcal{F}_{t}}:=\sigma\{u_{0}\vee W(s), s\leq t\}$;

\item[(ii)] $u$ satisfies (i),(ii) and (iii) of Definition 4.1.
\end{itemize}}
\end{defi}

Finally we define the martingale solutions.
For any fixed $T>0$, let $\Omega^T:=C([0,T]; H^{-1})$ be the space of all continuous functions from $[0,T]$ to $H^{-1}$. \\
For $0\leq t\leq T$, define the filtration:
$$\mathcal{F}_{t}=\sigma\{x(r):0\leq r\leq t, x\in \Omega^T \}.$$

\begin{defi}[Martingale solution]
{\sl We say that a probability measure $P\in \mathcal{P}(C([0,T]; H^{-1}))$ is called a martingale solution of (1.2) with initial value $u_{0}$ if
\begin{itemize}
\item[(M1)] $P\bigl(u(0)=u_{0},u\in L^{\infty}(0,T; \tilde{H}^{0,1})\cap L^{2}(0,T; \tilde{H}^{1,1})\bigl)=1$, and
$$ P\{u\in C([0,T],H^{-1}):\int_{0}^{T}\| F(u(s))\|_{H^{-1}}ds+\int_{0}^{T}\| \sigma(s,u(s))\|_{L^{2}(\ell^2,H)}^2ds<+\infty\}=1;$$

\item[(M2)] For every $l\in C^{1}(\mathbb{T}^{2})$, the process
$$M_{l}(t,u)=\langle u(t),l \rangle-\int_{0}^{t}\langle F(u(s)),l\rangle ds$$
is a continuous square integrable $\mathcal{F}_{t}-martingale$ with respect to P, whose quadratic variation process is $\int_{0}^{t}\| \sigma^{*}(s,u(s))(l)\|_{\ell^2}^{2}ds$,
where the asterisk denotes the adjoint operator of $\sigma(s,u(s))$;

\item[(M3)] We have
$$E^{P}\Bigl(\sup\limits_{t\in[0,T]}\| u(t)\|_{L^{2}}^{2}+\int_{0}^{T}\| u(t)\|_{H^{1,0}}^{2}dt\Bigr)
\leq C_{T}(1+\| u_{0}\|_{L^{2}}^{2}).$$
\end{itemize}}
\end{defi}
\begin{Remark}

By the above definitions, we know immediately that if $u$ is a (probabilistically) strong solution with respect to the fixed cylindrical Wiener process $W$, $(u,W)$ is a (probabilistically)  weak solution.
Moreover, let $P$ denote the law of $u$ in $C([0,T],H^{-1})$, then $P$ is a martingale solution.

Notice that by martingale representation theorem, (see for example \cite{Da Prato1992Stochastic})  the existence of  martingale  solution can lead to the existence of (probabilistically) weak solution. And the law of the weak solution gives a martingale solution $P$.
\end{Remark}


\begin{defi}[Condition  C]
{\sl The diffusion coefficient $\sigma$ is a measurable mapping from $\bigl([0,T]\times \tilde{H}^{1,1},\mathcal{B}([0,T]\times \tilde{H}^{1,1})\bigl)$ to $\bigl(L^2(\ell^2,\tilde{H}^{1,1}),\mathcal{B}(L^{2}(\ell^2,\tilde{H}^{1,1}))\bigl)$ such that    :
\item[(i)]\textbf{Growth condition}\\
There exist nonnegative constants $K'_{i}$, $K_{i}$ and $\tilde{K}_{i}$ $(i=0,1,2)$ such that for every
$t\in[0,T]$ and $u\in \tilde{H}^{1,1}$ :
\beno
\begin{split}
\| \sigma(t,u) \|_{L^{2}(\ell^2,H^{-1})}^{2}\leq & K'_{0}+K'_{1}\| u \|_{L^{2}}^{2};\\
\| \sigma(t,u) \|_{L^{2}(\ell^2,H)}^{2}\leq &K_{0}+K_{1}\| u \|_{L^{2}}^{2}+K_{2}\| \partial_{1}u \|_{L^{2}}^{2};\\
\| \sigma(t,u) \|_{L^{2}(\ell^2,H^{0,1})}^{2}\leq &\tilde{K}_{0}
+\tilde{K}_{1}\| u \|_{0,1}^{2}+\tilde{K}_{2}(\| \partial_{1}u \|_{L^{2}}^{2}+\| \partial_{2}\partial_{1}u \|_{L^{2}}^{2}).
\end{split}
\eeno
\item[(ii)]\textbf{Lipschitz condition} \\
 There exist constants $L_{1}$ and $L_{2}$ such that for $t\in[0,T]$ and $u,v\in\tilde{H}^{1,1}$:
$$\| \sigma(t,u)- \sigma(t,v) \|_{L^{2}(\ell^2,H)}^{2}\leq
L_{1}\| u-v\|_{L^{2}}^{2}+L_{2}\| \partial_{1}(u-v) \|_{L^{2}}^{2}. $$}
\end{defi}
\begin{Remark}
A typical example of $\sigma$ satisfying Condition (\textbf{C}) is the following:\\
 First we recall the H\"{o}lder space $C^{k+\tau}$ (k is an nonnegative integer and $0\leq \tau<1$) as: $u$ has $k$th derivatives and
$$\| u \|_{C^{k+\tau}}:=\sum\limits_{|\alpha|\leq k}\| D^\alpha u \|
+\sum\limits_{|\alpha|= k}\sup\limits_{x\neq y} \frac{| D^\alpha u(x)- D^\alpha u(y)|}{|x-y|^\tau}<\infty.$$
For $u\in H^{1,1}$ and $y\in \ell^2$, let
$$\sigma(t,u)y=\sum\limits_{k=1}^{\infty}(c_{k}\partial_{1}u+b_{k}g(u))\langle y,\psi_{k}\rangle_{\ell^2},$$
where $\psi_{k}$, as defined in Section 4.1, is the orthonormal basis of $\ell^2$  and $c_{k}\in C^{\rho}$, $\sum\limits_{k=1}^{\infty}\|c_{k}(\xi)\|_{C^{\rho}}^{2}\leq M_{1}$ for some $\rho>2$,
$b_{k}\in L^{\infty}$, $\partial_2b_{k}\in L^{\infty}$, $\sum\limits_{k=1}^{\infty}|b_{k}(\xi)|^2\leq M_{2}$, and $\sum\limits_{k=1}^{\infty}|\partial_{2}b_{k}(\xi)|^2\leq M_{2}$, $\xi\in \mathbb{T}^{2}$.
We also assume that $\|g\|_{C^{1}}\leq C(g).$
 Here $C^{\rho}$ and $C^{1}$  are the H\"{o}lder spaces.
 And suppose that $\dive(c_{k}\partial_{1}u+b_{k}g(u))=0 $.
Then we have
\beno
\begin{split}
\| \sigma(t,u) \|_{L^{2}(\ell^2,H^{-1})}\leq& \bigl(\sqrt{M_{1}}+\sqrt{M_{2}}C(g)\bigl)\| u \|_{L^{2}}+\sqrt{M_{2}}C(g);\\
\| \sigma(t,u) \|_{L^{2}(\ell^2,H)}\leq& \sqrt{M_{1}}\| \partial_{1}u \|_{L^{2}}+\sqrt{M_{2}}\bigl(C(g)\| u \|_{L^{2}}+C(g)\bigl);\\
\| \sigma(t,u) \|_{L^{2}(\ell^2,H^{0,1})}\leq &  \sqrt{M_{1}}\| \partial_{1}u \|_{L^{2}}+\sqrt{M_{2}}\bigl(C(g)\| u \|_{L^{2}}+C(g)\bigl)\\
&+(\sum\limits_{k=1}^{\infty}\|\partial_{2}(c_{k}\partial_{1}u)\|_{L^{2}}^{2})^{\frac{1}{2}}
+(\sum\limits_{k=1}^{\infty}\|\partial_{2}(b_{k}g(u))\|_{L^{2}}^{2})^{\frac{1}{2}}\\
\leq &  \sqrt{M_{1}}\| \partial_{1}u \|_{L^{2}}+\sqrt{M_{2}}\bigl(C(g)\| u \|_{L^{2}}+C(g)\bigl)\\
&+\sqrt{M_{1}}(\| \partial_{1}u \|_{L^{2}}+\| \partial_{1}\partial_{2}u \|_{L^{2}})
+\sqrt{M_{2}}C(g)(\| u \|_{L^{2}}+\| \partial_{2}u \|_{L^{2}});\\
\| \sigma(t,u)- \sigma(t,v) \|_{L^{2}(\ell^2,H)}\leq & \sqrt{M_{1}}\| \partial_{1}(u-v) \|_{L^{2}}+\sqrt{M_{2}}C(g)\|u-v \|_{L^{2}},
\end{split}
\eeno
where the first inequality above is due to (2) in page 140, section 2.8.2
\cite{Triebel1983}.

\end{Remark}

\subsection{Main theorems of stochastic cases} In this section we state two theorems about the well-posedness of equation (1.2), which will be proved in the following sections.
\begin{thm}\label{thm2}
{\sl Assume that $u_{0}$ is a random variable in $L^{4}(\Omega,\tilde{H}^{0,1})$ and suppose that $\sigma$ satisfies condition (\textbf{C}) with $K_{2}<\frac{2}{11}$ and
$\tilde{K}_{2}<\frac{2}{5}$. Then (1.2) has a global martingale solution.}
\end{thm}

\begin{thm}[Pathwise uniqueness]\label{thm3}
{\sl Assume that $u_{0}$ is a random variable in $L^{4}(\Omega,\tilde{H}^{0,1})$. Suppose that $\sigma$ satisfies condition (\textbf{C}) with $K_{2}<\frac{2}{11}$,
$\tilde{K}_{2}<\frac{2}{5}$ and $L_{2}<\frac{2}{5}$. If $u,v$ are two weak solutions on the same stochastic basis $(\Omega, \mathcal{F},P)$. Then we have $u=v$     $P- a.s.$}
\end{thm}

\begin{Remark}
By the Yamada-Watanabe theorem, (cf. \cite{liu2015stochastic}) the existence of (probabilistically) weak solution and pathwise uniqueness lead to the existence of the (probabilistically)strong solution.

\end{Remark}
\subsection{Galerkin Approximation and A Priori Estimates.}
From now on we use $C$ to denote the constant which can be different from line to line.
%

Fix $ n\geq 1$ and consider the following stochastic ordinary differential equations on $\mathcal{H}_{n}$ :
$$u_{n}(0)=P_{n}u_{0},$$
and for $t \in [0,T]$ , $v \in \mathcal{H}_{n}$:

\begin{equation}\label{S4eq9} d(u_{n}(t),v)=\langle P_{n}F(u_{n}(t)),v\rangle dt+(P_{n}\sigma (t,u_{n}(t))\Pi_{n} dW(t),v).
 \end{equation}
Then for $k=1,...,n$ we have for $t\in [0,T]$:
\begin{equation*} d(u_{n}(t),e_{k})=\langle P_{n}F(u_{n}(t)),e_{k}\rangle dt+\sum \limits_{j=1}^{n} (P_{n}\sigma (t,u_{n}(t))\psi_{j},e_{k})d\beta_{j}(t).
\end{equation*}
Now we use \cite{liu2015stochastic} Thm 3.1.1 about existence and uniqueness of solutions to stochastic differential equations.
Note that since it is in finite dimensions,  there exists some constant $C(n)$ such that $\| v \|_{H^{2}}\leq C(n) \| v \|_{L^{2}}$ for $v\in \mathcal{H}_{n}$.\\
 Let $\varphi , \psi , v \in \mathcal{H}_{n}$; integration by parts implies that
\begin{equation*}\mid \langle \partial_{1}^{2} \varphi- \partial_{1}^{2} \psi , v \rangle \mid \leq \| \varphi-\psi \|_{1,0}\| v \|_{1,0}\leq C(n)^{2}\|\varphi-\psi\|_{L^{2}}\| v \|_{L^{2}}.\end{equation*}
Moreover, we have
\begin{equation*}
\begin{split}
\mid \langle B(\varphi)-B(\psi),v \rangle\mid &=\mid -\langle B(\varphi-\psi ,v),\varphi \rangle-\langle B(\psi ,v),\varphi-\psi \rangle\mid \\
&\leq C\| \varphi-\psi \|_{H^{1,0}}(\| \varphi \|_{H^{1,0}}+\| \psi \|_{H^{1,0}}) \| v \|_{H^{1,1}}\\
&\leq CC(n)^{3}\| \varphi-\psi \|_{L^{2}}(\| \varphi \|_{L^{2}}+\| \psi \|_{L^{2}})\| v\|_{L^{2}}.
\end{split}
\end{equation*}

Hence we know that for $u ,v \in \mathcal{H}_{n}$, and $\| u \|_{L^{2}},\| v \|_{L^{2}}\leq R$,
$$\mid \langle F(u)-F(v),u-v \rangle\mid \leq 2RC(n)^{3}\| u-v \|_{L^{2}}^{2},$$
The condition (\textbf{C}) implies that for $u ,v \in \mathcal{H}_{n}$, and $\| u \|_{L^{2}},\| v \|_{L^{2}}\leq R$,
\begin{equation*}
\begin{split}
\| P_{n}(\sigma(t,u)-\sigma(t,v))\|_{L^2(\ell^2,H)}^{2}
&\leq \| \sigma(t,u)-\sigma(t,v)\|_{L^2(\ell^2,H)}^{2}\\
&\leq L_{1}\| u-v\|_{L^{2}}^{2}+L_{2}\| \partial_{1}(u-v) \|_{L^{2}}^{2}\\
&\leq C(n)^{2}(L_{1}+L_{2})\| u-v \|_{L^{2}}^{2}.
\end{split}
\end{equation*}
So it satisfies local weak monotonicity.
Moreover,
\begin{equation*}
\begin{split}
2\langle u,P_{n}F(u)\rangle+\|P_{n}\sigma(t,u)\|_{{L^2(\ell^2,H)}}^{2}&\leq \| u\|_{H^{1,0}}^{2}+\| \sigma \|_{L^{2}(\ell^2,H)}^{2}\\
&\leq C(n)^2\| u \| _{L^{2}}^{2}+K_{0}+K_{1}\| u \|_{L^{2}}^{2}+K_{2}\| \partial_{1}u \|_{L^{2}}^{2}\\
&\leq K_{0}+\bigl(C(n)^2+K_{1}+K_{2}C(n)^2\bigl)\| u \| _{L^{2}}^{2}.
\end{split}
\end{equation*}
Thus it satisfies weak coercivity.\\
Hence by \cite{liu2015stochastic} Thm 3.1.1, there exists a unique global strong  solution $u_{n}(t)$ to (4.1). Moreover, $u\in C([0,T],\mathcal{H}_{n}), P-a.s.$

\subsection{The $L^{2}$ Energy Estimates}
In this section, we give the following a priori estimates.
\begin{lem}

We have the following energy estimates under the hypothesis of Thm 4.1:
$$E(\sup\limits_{t\in[0,T]}\| u_{n}(t) \|_{L^{2}}^{2})+E\int_{0}^{T}\| u_{n}(t) \|_{H^{1,0}}^{2}dt\leq C(1+E\| u_{0} \|_{L^{2}}^{2}).$$
\end{lem}
\proof  Let $u_{n}(t)$ be the  solution to (4.1) described above. By It\^{o}'s formula, we have:
\beq \label{S4eq2}
\begin{split}
\| u_{n}(t) \|_{L^{2}}^{2}=&\| P_{n}u_{0}\|_{L^{2}}^{2} +2\int_{0}^{t}(\sigma(s,u_{n}(s))dW_{n}(s),u_{n}(s))\\
 &-2\int_{0}^{t}\| \partial_{1}u_{n}(s)\|_{L^{2}}^{2}ds+\int_{0}^{t}\| P_{n}\sigma(s,u_{n}(s))\Pi_{n}\|_{L^{2}(\ell^2,H)}^{2}ds.
 \end{split}
\eeq
The growth condition implies that
\begin{equation}\label{S4eq3}
 \int_{0}^{t}\| P_{n}\sigma(s,u_{n}(s))\Pi_{n}\|_{L^{2}(\ell^2,H)}^{2}ds\leq \int_{0}^{t}[K_{0}+K_{1}\| u_{n}(t) \|_{L^{2}}^{2}+K_{2}\| \partial_{1}u_{n}(t) \|_{L^{2}}^{2}]ds.
 \end{equation}
The Burkholder-Davis-Gundy inequality( see Thm 6.1.2, chapter 6 in \cite{liu2015stochastic})  and the Young inequality as well as the growth condition imply that:
\begin{equation}\label{S4eq4}
\begin{split}
&E\Bigl(\sup\limits_{s\leq t}\mid 2\int_{0}^{s}(P_{n}\sigma(r,u_{n}(r))dW_{n}(r),u_{n}(r))\mid\Bigr)\\
&\leq 4E\Bigl\{\int_{0}^{t}\| P_{n}\sigma(r,u_{n})(r)\Pi_{n}\|_{L^{2}(\ell^2,H)}^{2}\| u_{n}(r) \|_{L^{2}}^{2}dr\Bigr\}^{\frac{1}{2}}\\
&\leq \beta E(\sup\limits_{s\leq t} \| u_{n}(s) \|_{L^{2}}^{2})+\frac{4}{\beta}E\int_{0}^{t}[K_{0}+K_{1}\| u_{n}(s) \|_{L^{2}}^{2}+K_{2}\| \partial_{1}u_{n}(t) \|_{L^{2}}^{2}]ds.
\end{split}
\end{equation}
 Since $K_{2}<\frac{2}{11}$,  we can choose $0<\beta<1$ such that $(\frac{4}{\beta}+1)K_{2}-2<0$.

 By \eqref{S4eq2}-\eqref{S4eq4} and dropping some negative terms, we deduce:
\begin{equation*}
(1-\beta)E\sup\limits_{s\in[0,t]}\| u_{n}(s) \|_{L^{2}}^{2}\leq E\| u(0)\|_{L^{2}}^{2}+CK_{0}T+CE\int_{0}^{t}K_{1}\| u_n(s)\|_{L^{2}}^{2}ds.
\end{equation*}
Gronwall's lemma implies that
\begin{equation}\label{S4eq40}
 E(\sup\limits_{t\in[0,T]}\| u_{n}(t) \|_{L^{2}}^{2})\leq C,
 \end{equation}
where C is a constant depending on $K_{0},K_{1},K_{2},T$ but not $n$.\\
Inserting \eqref{S4eq40} back to \eqref{S4eq2}-\eqref{S4eq4} yields
\begin{equation*}
E(\sup\limits_{t\in[0,T]}\| u_{n}(t) \|_{L^{2}}^{2})+E\int_{0}^{t}\| u_{n}(t) \|_{H^{1,0}}^{2}ds\leq C(1+E\| u_{0} \|_{L^{2}}^{2}),
\end{equation*}
where C is a constant depending on $K_{0},K_{1},K_{2},T$ but not $n$.\\
This completes the proof.
\qed
\\

However, it is not enough that we only have $L^{2}(\Omega)$ estimates. We also need an $L^{4}(\Omega)$ uniform
estimates of $u_{n}$.
\begin{lem}
We have the following uniform estimates under the hypothesis of Thm 4.1:
\begin{equation*}
E(\sup\limits_{t\in[0,T]}\| u_{n}(t)\|_{L^{2}}^{4})+E\int_{0}^{T}\| u_{n}(t)\|_{L^{2}}^{2}\| u_{n}(t)\|_{H^{1,0}}^{2}dt \leq C(1+E\| u_{0} \|_{L^{2}}^{4}).
\end{equation*}
\end{lem}
\proof
Applying once more the It\^{o}'s formula to the square of $\| \cdot \|_{L^{2}}^{2}$, we obtain:
\begin{equation}\label{S4eq5}
\| u_{n}(t)\|_{L^{2}}^{4}=\| P_{n}u_{0} \|_{L^{2}}^{4}-4\int_{0}^{t}\| \partial_{1} u_{n}(s)\|_{L^{2}}^{2}\| u_{n}(s)\|_{L^{2}}^{2}ds+I_{1}+I_{2}+I_{3},
\end{equation}
where
\beno
\begin{split}
I_{1}=&4\int_{0}^{t}(\sigma(s,u_{n}(s))dW_{n}(s),u_{n}(s))\| u_{n}(s)\|_{L^{2}}^{2}\ , \\
I_{2}=&2\int_{0}^{t}\| P_{n}\sigma(s,u_{n}(s))\Pi_{n}\|_{L^{2}(\ell^2,H)}^{2}\| u_{n}(s) \|_{L^{2}}^{2}\,ds,\\
I_{3}=&4\int_{0}^{t}\| (P_{n}\sigma(s,u_{n}(s))\Pi_{n})^{*}(u_{n})\|_{l^{2}}^{2}ds.
\end{split}
\eeno
The growth condition implies that
\begin{equation}\label{S4eq6}
I_{2}(t)+I_{3}(t)\leq 6\int_{0}^{t}(K_{0}+K_{1}\| u_{n}(s) \|_{L^{2}}^{2}+K_{2}\| \partial_{1}u_{n}(t) \|_{L^{2}}^{2})\| u_{n}(s) \|_{L^{2}}^{2}ds.
\end{equation}
The Burkholder-Davis-Gundy inequality, the growth condition and the Young inequality imply that:
\begin{equation}\label{S4eq7}
\begin{split}
E(\sup\limits_{s\leq t}I_{1}(s))\leq& 8 E\Bigl\{\int_{0}^{t}\| \sigma(r,u_{n}(r))
\|_{L^{2}(\ell^2,H)}^{2}\| u_{n}(r)\|_{L^{2}}^{6}dr\Bigr\}^{\frac{1}{2}}\\
\leq &\gamma E(\sup\limits_{s\leq t}\| u_{n}(s)\|_{L^{2}}^{4})\\
&+\frac{16}{\gamma}E\int_{0}^{t}(K_{0}+K_{1}\| u_{n}(s)\|_{L^{2}}^{2}+K_{2}\| \partial_{1}u_{n}(t) \|_{L^{2}}^{2})\| u_{n}(s)\|_{L^{2}}^{2}ds.
\end{split}
\end{equation}
Since $K_{2}<\frac{2}{11}$, we can choose $0<\gamma<1$, such that
$6K_{2}+ \frac{16}{\gamma}K_{2}-4<0$.\\
Thus combining \eqref{S4eq5}-\eqref{S4eq7} and dropping some negative terms on the right of the inequality, we have:
$$(1-\gamma)E(\sup\limits_{t\in[0,T]}\| u_{n}(t) \|_{L^{2}}^{4})\leq
E\| u(0) \|_{L^{2}}^{4}+E\int_{0}^{t}C_{1}\| u_{n}(s) \|_{L^{2}}^{4}+C_{2}\| u_n(s) \|_{L^{2}}^{2}ds.
$$

Since we have obtained $E(\sup\limits_{t\in[0,T]}\| u_{n}(t) \|_{L^{2}}^{2})\leq C$,
the Gronwall inequality yields
$$E(\sup\limits_{t\in[0,T]}\| u_{n}(t)\|_{L^{2}}^{4})<\infty.$$
Similar as in the proof of Lemma 4.1, we complete the proof.
\qed

\subsection{Tightness and the Skorokhod Theorem}
 In this section we use the classical tightness methods. Similar to the deterministic cases, $L^{2}$-estimates are not enough to obtain strong convergence. As a result, we use tightness in the following space $\mathcal{X}$.

Let $\hat{P}_{n}$ be the law of $u_{n}$ on $C([0,T];H^{-1})$.
\begin{lem}Under the hypothesis of Thm. 4.1,
$\hat{P}_{n}$ is tight in the space
$$\mathcal{X}=C([0,T];H^{-1})\cap L^{2}([0,T];H) \cap L_{w}^{2}([0,T];H^{1,1})\cap L_{w^*}^{\infty}([0,T];H^{0,1}),$$
where $L_{w}^{2}([0,T];H^{1,1})$ denotes $L^{2}([0,T]; H^{1,1})$ with the weak topology and $L_{w^*}^{\infty}([0,T];H^{0,1})$ denotes $L^{\infty}([0,T];H^{0,1})$ with the weak star topology.
\end{lem}
\proof
Firstly, since $\tilde{K}_{2}<\frac{2}{5}$, we can choose $\tilde{\alpha}, \tilde{\beta}\in (0,1)$, such that:
$$\tilde{K}_{2}+2\tilde{\alpha}+\frac{4}{\tilde{\beta}}\tilde{K}_{2}<2.
$$
From the calculation in Lemma 3.2, by the Young inequality, we deduce that:
\begin{equation}\label{S4eq10}
\mid(\partial_{2}(u\cdot \nabla u )\mid \partial_{2}u)\mid\leq \tilde{\alpha}\|\partial_{1}\partial_{2}u\|_{L^{2}}^{2}+C(\tilde{\alpha})\| \partial_{1}u\|_{L^{2}}^{2}\| \partial_{2}u\|_{L^{2}}^{2}.
\end{equation}

Let \beno
\begin{split}
K_{R}:=&\Bigl\{ u\in C([0,T], H^{-1});\sup\limits_{0\leq t\leq T}\| u(t) \|_{L^{2}}^{2}+\int_{0}^{T}\| u(t) \|_{H^{1,0}}^{2}dt+\| u \|_{C^{\frac{1}{8}}([0,T];H^{-1})}\\
&+
\sup\limits_{0\leq t\leq T}e^{-2C(\tilde{\alpha})\int_{0}^{t}\| \partial_{1}u\|_{L^{2}}^{2}ds}  \| u(t) \|_{H^{0,1}}^{2}+
\int_{0}^{T}e^{-2C(\tilde{\alpha})\int_{0}^{t}\| \partial_{1}u\|_{L^{2}}^{2}dt}\| u \|_{H^{1,1}}^{2}dt\leq R\Bigr\}.
\end{split}
\eeno

Now we want to show that
\begin{itemize}
\item[(i)] For any $R>0$,$K_{R}$ is relatively compact in $\mathcal{X}$;

\item[(ii)] For any $\epsilon >0$, there exists $R>0$, such that $\hat{P}_{n}(K_{R})>1-\epsilon$ for any $n$.
\end{itemize}

\noindent Proof of (i): By the definition of $K_{R}$, it is obvious that $u\in K_{R}$ is bounded in $L^{2}([0,T];H^{1,1})$, thus $K_{R}$ is relatively compact in  $L_{w}^{2}([0,T];H^{1,1})$ and $L_{w^*}^{\infty}([0,T];H^{0,1})$.\\
Moreover, by definition, $K_{R}$ is equicontinuous in $C([0,T];H^{-1})$. The compactness in $C([0,T];H^{-1})$  can be obtained by Arzela-Ascoli Lemma.  \\
Finally we prove the compactness in $L^{2}([0,T]; L^{2})$.
Let $u_{n}$ be a sequence in $K_{R}$. We can assume that $u_{n}$ converges to $u$ in $C([0,T]; H^{-1})\cap L_{w}^{2}([0,T]; H^{1,1})$. Then we have:
\begin{equation*}
\begin{split}
\int_{0}^{T}\| u_{n}-u\|_{L^{2}}^{2}dt&\leq \int_{0}^{T}\| u_{n}-u\|_{H^{1}}\| u_{n}-u\|_{H^{-1}}dt\\
&\leq \bigl(\int_{0}^{T}\| u_{n}-u\|_{H^{1}}^{2}\,dt\bigr)^{\frac{1}{2}}\bigl(\int_{0}^{T}\| u_{n}-u\|_{H^{-1}}^{2}\,dt\bigr)^{\frac{1}{2}}\\
&\leq C_{R,T}\sup\limits_{t\in[0,T]}\| u_{n}-u\|_{H^{-1}}^{2}\\
&\rightarrow 0,
\end{split}
\end{equation*}
which finishes the proof of (i).\\
\\
Proof of (ii):
By Lemma 4.1 as well as Chebyshev's inequality, we can choose $R_{0}$ large enough such that:
\begin{equation}\label{S4eq11}
P\Bigl(\sup\limits_{t\in[0,T]}\| u_{n}(t) \|_{L^{2}}^{2}+\int_{0}^{T}\| u_{n}(t) \|_{H^{1,0}}^{2}\,dt> \frac{R_{0}}{4}\Bigr)<\frac{\epsilon}{4}.
\end{equation}
Set $h(t)=2C(\tilde{\alpha})\int_{0}^{t}\| \partial_{1}u_{n} \|_{L^{2}}^{2}ds$.
Now we need another estimate as following:
\begin{equation}\label{S4eq12}
\begin{split}
&E(\sup\limits_{t\in[0,T]}(e^{-h(t)}\| u_{n}(t)\|_{H^{0,1}}^{2}))+E\int_{0}^{T}e^{-h(t)} \| u_{n}(t)\|_{H^{1,1}}^{2}dt\\
\leq &C(\tilde{K_{0}},\tilde{K_{1}},\tilde{K_{2}},T)(1+E\| u_{0}\|_{H^{0,1}}^{2}),
\end{split}
\end{equation}
the proof of which is postponed later to Lemma 4.4.

By \eqref{S4eq12} and Chebyshev's Inequality, we can choose $R_{0}$ large enough such that:
\begin{equation}\label{S4eq22}
\hat{P}_{n}\Bigl(\sup\limits_{0\leq t\leq T}e^{-2C(\tilde{\alpha})\int_{0}^{t}\| \partial_{1}u\|_{L^{2}}^{2}ds}  \| u(t) \|_{H^{0,1}}^{2}+\\
\int_{0}^{T}e^{-2C(\tilde{\alpha})\int_{0}^{t}\| \partial_{1}u\|_{L^{2}}^{2}ds}\| u \|_{H^{1,1}}^{2}dt>\frac{R_{0}}{4}\Bigl)<\frac{\epsilon}{4}.
\end{equation}
Now we fix $R_{0}$ and set
\begin{equation*}
\begin{split}
\hat{K}_{R_{0}}=\{&u\in  C([0,T], H^{-1}); \sup\limits_{t\in[0,T]}\| u(t) \|_{L^{2}}^{2}+\int_{0}^{T}\| u(t) \|_{H^{1,0}}^{2}\,dt\leq \frac{R_{0}}{4} \text{ and }\\
 &\sup\limits_{0\leq t\leq T}e^{-2C(\tilde{\alpha})\int_{0}^{t}\| \partial_{1}u\|_{L^{2}}^{2}ds}  \| u(t) \|_{H^{0,1}}^{2}+
\int_{0}^{T}e^{-2C(\tilde{\alpha})\int_{0}^{t}\| \partial_{1}u\|_{L^{2}}^{2}ds}\| u \|_{H^{1,1}}^{2}dt)\leq\frac{R_{0}}{4}\}.
\end{split}
\end{equation*}
Then we know $\hat{P}_{n}(C([0,T], H^{-1})\setminus \hat{K}_{R_{0}})<\frac{\epsilon}{2}.$
Now we only consider $u$ in $\hat{K}_{R_{0}}$.
By H\"{o}lder's inequality, we have:
\begin{equation}\label{S4eq18}
\begin{split}
E^{\hat{P}_{n}}\Bigl[\sup\limits_{s\neq t\in[0,T]}\bigl(\frac{\| \int_{s}^{t}-\partial_{1}^{2}u(r)+\dive(u\otimes u)dr\|_{H^{-1}}^{2}}{\mid t-s \mid}\bigr)1_{u\in \hat{K}_{R_{0}}}\Bigr]\\
\leq E^{\hat{P}_{n}}\Bigl[\int_{0}^{T}\|-\partial_{1}^{2}u(r)+\dive(u\otimes u)\|_{H^{-1}}^{2}\,dr1_{u\in \hat{K}_{R_{0}}}\Bigr].
\end{split}
\end{equation}
 The boundedness of $u$ in $L^{2}([0,T];H^{1,1})$ leads to the  boundedness of $\partial_{1}^{2}u$ in $L^{2}([0,T];H^{-1})$. By the definition of $K_{R}$, $u$ is also bounded in $L^{\infty}([0,T];H^{0,1})$. By  interpolation, $u$ is bounded in $L^{4}([0,T];H^{\frac{1}{2}})$. By Sobolev imbedding, $u$ is bounded in $L^{4}([0,T];L^{4})$. Thus we obtain:
 \begin{equation}\label{S4eq19}
 E^{\hat{P}_{n}}\Bigl[\int_{0}^{T}\|-\partial_{1}^{2}u(r)+div(u\otimes u)\|_{H^{-1}}^{2}\,dr1_{u\in \hat{K}_{R_{0}}}\Bigr]\leq C(R_{0}),
 \end{equation}
 where $C(R_{0})$ is independent of $n$.\\
 Thus by \eqref{S4eq18} and \eqref{S4eq19}, we have
 \begin{equation}\label{S4eq20}
E^{\hat{P}_{n}}\Bigl[\sup\limits_{s\neq t\in[0,T]}\bigl(\frac{\| \int_{s}^{t}-\partial_{1}^{2}u(r)+div(u\otimes u)dr\|_{H^{-1}}^{2}}{\mid t-s \mid}\bigr)1_{u\in \hat{K}_{R_{0}}}\Bigr]\leq C(R_{0}).
\end{equation}
 Moreover, for any $T\geq t>s\geq 0$ and any $p\in \mathbb{N}$ we have
 \begin{equation*}
 \begin{split}
  E^{\hat{P}_{n}}\| \int_{s}^{t}P_{n}\sigma(r,u(r))dW_{n}(r)\|_{H^{-1}}^{2p}&\leq C_{p}E^{\hat{P}_{n}}\Bigl(\int_{s}^{t}\| \sigma(r,u(r))\|_{L^{2}(\ell^2,H^{-1})}^{2}\,dr\Bigr)^{p}\\
  &\leq C_{p}\mid t-s\mid^{p-1}\int_{s}^{t}E^{\hat{P}_{n}}\| \sigma(r,u(r))\|_{L^{2}(\ell^2,H^{-1})}^{2p}dr\\
  &\leq C_{p}\mid t-s\mid^{p-1}\int_{s}^{t}E^{\hat{P}_{n}}(\| u(r)\|_{L^{2}}^{2p}+1)dr\\
  &\leq C_{p,T}\mid t-s\mid^{p}\bigl(1+E(\sup\limits_{t\in[0,T]}\| u_{n}(t)\|_{L^{2}}^{2p})\bigr).
 \end{split}
 \end{equation*}
 By Kolmogorov's criterion, for any $\alpha \in (0,\frac{p-1}{2p})$, we have:
 \begin{equation}\label{S4eq21}
 E^{\hat{P}_{n}}\Bigl(\sup\limits_{s\neq t\in[0,T]}\frac{\| \int_{s}^{t}P_{n}\sigma(r,u(r))dW_{n}(r)\|_{H^{-1}}^{2p}}{\mid t-s\mid^{p\alpha}}\Bigr)\leq C_{p,T}\bigl(1+E(\sup\limits_{t\in[0,T]}\| u_{n}(t)\|_{L^{2}}^{2p})\bigr).
 \end{equation}
 Choose $p=2$. By \eqref{S4eq20} and \eqref{S4eq21} , we get for $\alpha=\frac{1}{8}$:
 \begin{equation*}
 E^{\hat{P}_{n}}\Bigl(\sup\limits_{s\neq t\in[0,T]}\frac{\| u(t)-u(s)\|_{H^{-1}}}{\mid t-s\mid^{\alpha}}1_{u\in \hat{K}_{R_{0}}}\Bigr)<\infty.
 \end{equation*}
 Similarly, we choose $R>R_{0}$ large enough and obtain:
 \begin{equation}\label{S4eq45}
 \hat{P}_{n}(\|u \|_{C^{\frac{1}{8}}([0,T];H^{-1})}>\frac{R}{4} \text{ and } u\in \hat{K}_{R_{0}})<\frac{\epsilon}{4}.
 \end{equation}
 Combining with \eqref{S4eq11},\eqref{S4eq22} and \eqref{S4eq45} complete the proof.
 \qed

\begin{lem}
Under the hypothesis of Thm. 4.1, the uniform estimates \eqref{S4eq12} holds.

\end{lem}
\proof
Using again the It\^{o}'s Formula to $e^{-h(t)}\| u_{n}(t) \|_{H^{0,1}}^{2}$, we obtain:
\begin{equation}\label{S4eq13}
\begin{split}
e^{-h(t)}\| u_{n}(t)\|_{H^{0,1}}^{2}&=
\| P_{n}u(0)\|_{H^{0,1}}^{2}+\sum\limits_{j=1}^{3}T_{j}(t)-\int_{0}^{t}e^{-h(s)}h'(s)\| u_{n}(s)\|_{H^{0,1}}^{2}ds\\
&+\int_{0}^{t}e^{-h(s)}[-2\| \partial_{1}u_{n}(s)\|_{L^{2}}^{2}-2\| \partial_{1}\partial_{2}u_{n}(s)\|_{L^{2}}^{2}]ds,
\end{split}
\end{equation}
where
\beno
\begin{split}
T_{1}(t)=&-2\int_{0}^{t}e^{-h(s)}\langle \partial_{2}(u_{n}\cdot\nabla u_{n}),\partial_{2}u_{n}(s)\rangle,\\
T_{2}(t)=&2\int_{0}^{t}e^{-h(s)}(\sigma(s,u_{n}(s))dW_{n}(s),u_{n}(s))_{H^{0,1}},\\
T_{3}(t)=&\int_{0}^{t}e^{-h(s)}\| P_{n}\sigma(s,u_{n}(s))\Pi_{n}\|_{L^{2}(\ell^2,H^{0,1})}^{2}\,ds.
\end{split}
\eeno
The growth condition implies that
\begin{equation}\label{S4eq14}
T_{3}(t)\leq \int_{0}^{t}e^{-h(s)}\bigl[\tilde{K_{0}}+\tilde{K_{1}}\| u_{n}(s)\|_{H^{0,1}}^{2}+\tilde{K_{2}}\| (\partial_{1}u_{n}(s)\|_{L^{2}}^{2}+\| \partial_{1}\partial_{2}u_{n}(s)\|_{L^{2}}^{2})\bigr]\,ds.
\end{equation}
For $T_{1}(t)$, we use \eqref{S4eq10}:
\begin{equation}\label{S4eq15}
\mid T_{1}(t)\mid\leq \int_{0}^{t}e^{-h(s)}\bigl[2\tilde{\alpha}\| \partial_{1}\partial_{2}u_{n}\|_{L^{2}}^{2}+2C(\tilde{\alpha})\| \partial_{1}u_{n}\|_{L^{2}}^{2}\| \partial_{2}u_{n}\|_{L^{2}}^{2}\bigr]\,ds.
\end{equation}
Similar to \eqref{S4eq7}, we have
\begin{equation}\label{S4eq16}
\begin{split}
&E\Bigl(\sup\limits_{s\leq t}\mid 2\int_{0}^{s}e^{-h(r)}(\sigma(r,u_{n}(r))dW_{n}(r),u_{n}(r))_{H^{0,1}}\mid\Bigr)\\
&\leq 4E\Bigl\{\int_{0}^{t}e^{-h(r)}\| P_{n}\sigma(r,u_{n})(r)\Pi_{n}\|_{L^{2}(\ell^2,H^{0,1})}^{2}e^{-h(r)}\| u_{n}(r) \|_{H^{0,1}}^{2}dr\Bigr\}^{\frac{1}{2}}\\
&\leq \tilde{\beta} E(\sup\limits_{s\leq t}(e^{-h(s)} \| u_{n}(s) \|_{H^{0,1}}^{2}))\\
&+\frac{4}{\tilde{\beta}}E\int_{0}^{t}e^{-h(s)}\bigl[\tilde{K_{0}}+\tilde{K_{1}}\| u_{n}(s) \|_{H^{0,1}}^{2}+\tilde{K_{2}}(\| \partial_{1}u_{n}(s)\|_{L^{2}}^{2}+\| \partial_{1}\partial_{2}u_{n}(s)\|_{L^{2}}^{2})\bigr]\,ds.
\end{split}
\end{equation}
Combining \eqref{S4eq13}-\eqref{S4eq16} and dropping some negative terms, we have:
\begin{equation*}
\begin{split}
&(1-\tilde{\beta})E(\sup\limits_{t\in[0,T]}e^{-h(t)}\| u_{n}(t)\|_{H^{0,1}}^{2})\\
&\leq E\| P_{n}u_{0}\|_{H^{0,1}}^{2}+E\int_{0}^{T}e^{-h(s)}(1+\frac{4}{\tilde{\beta}})(\tilde{K_{0}}+\tilde{K_{1}}\| u_{n}(s) \|_{H^{0,1}}^{2})ds\\
\end{split}
\end{equation*}
By Gronwall's inequality,
\begin{equation}\label{S4eq17}
E(\sup\limits_{t\in[0,T]}(e^{-h(t)}\| u_{n}(t)\|_{H^{0,1}}^{2}))\leq C(T,\tilde{K_{0}},\tilde{K_{1}},\tilde{K_{2}})\| u_{0}\|_{H^{0,1}}^{2}
\end{equation}
Combining \eqref{S4eq13}-\eqref{S4eq16} again with the estimate \eqref{S4eq17} we obtain:
\begin{equation*}
\begin{split}
E(\sup\limits_{t\in[0,T]}(e^{-h(t)}\| u_{n}(t)\|_{H^{0,1}}^{2}))+&E\int_{0}^{T}e^{-h(t)} \| u_{n}(t)\|_{H^{1,1}}^{2}dt\\
\leq &C(\tilde{K_{0}},\tilde{K_{1}},\tilde{K_{2}},T)(1+E\| u_{0}\|_{H^{0,1}}^{2}).
\end{split}
\end{equation*}
\qed

The classical Skorokhod Theorem can only be used in metric space. We will use the following Jakubowski's  version of the Skorokhod Theorem in the form given by Brze\'{z}niak and Ondrej\'{a}t \cite{brzezniak2013stochastic} Thm A.1. and it was proved by A. Jakubowski in \cite{Jakubowski1998Short}.
\begin{thm}
{\sl Let $\mathcal{Y}$ be a topological space such that there exists a sequence ${f_{m}}$ of continuous functions $ f_{m}:\mathcal{Y}\rightarrow \mathbb{R}$ that separates points of $\mathcal{Y}$. Let us denote by $\mathcal{S}$ the $\sigma$-algebra generated by the maps ${f_{m}}$. Then
\begin{itemize}
\item[(j1)] every compact subset of  $\mathcal{Y}$ is metrizable;

\item[(j2)] if $(\mu_{m})$ is tight sequence of probability measures on$(\mathcal{Y},\mathcal{S})$, then there exists a subsequence $(m_{k})$, a probability space $(\Omega , \mathcal{F},P)$ with $\mathcal{Y}$-valued Borel measurable variables $\xi_{k}$,$\xi$ such that $\mu_{m_{k}}$ is the law of $\xi_{k}$ and $\xi_{k}$ converges to $\xi$ almost surely on $\Omega$. Moreover, the law of $\xi$ is a Radon measure.
    \end{itemize}}
\end{thm}
Now we check the $\mathcal{X}$ defined in Lemma 4.3 satisfies the above condition.
It is sufficient to prove that on each space appearing in the definition of $\mathcal{X}$ there exists a countable set of continuous real-valued functions separating points:\\
Since $C([0,T];H^{-1})$ and $L^{2}([0,T];H)$ are separable Banach spaces, it is easy to see the condition in Thm. 4.3 is satisfied. \\
For the space $L_{w}^{2}([0,T];H^{1,1})$ it is sufficient to put
$$f_{m}(u):=\int_{0}^{T}(u(t),v_{m}(t))_{H^{1,1}}dt\in\mathbb{R}, u\in L_{w}^{2}([0,T];H^{1,1}), m\in \mathbb{N},$$
where $v_{m}$ is a dense subset of $L^{2}([0,T];H^{1,1})$.\\
Similarly for the space $L_{w^*}^{\infty}([0,T];H^{0,1})$ it is sufficient to put
$$f_{m}(u):=\int_{0}^{T}(u(t),v_{m}(t))_{H^{0,1}}dt\in\mathbb{R}, u\in L_{w^*}^{\infty}([0,T];H^{0,1}), m\in \mathbb{N},$$
where $v_{m}$ is a dense subset of $L^{1}([0,T];H^{0,1})$.\\
Now all the conditions of the above Skorokhod theorem are satisfied. By Thm 4.3, there exists another probability space $(\tilde{\Omega},\tilde{\mathcal{F}},\tilde{\mathbb{P}})$ and a subsequence ${\hat{P}_{n_{k}}}$ as well as random variables ${\tilde{u}_{n_{k}}}$ in the space $(\tilde{\Omega},\tilde{\mathcal{F}},\tilde{\mathbb{P}})$, such that
\begin{itemize}
\item[(i)] ${\tilde{u}_{n_{k}}}$ has the law $\hat{P}_{n_{k}}$;
\item[(ii)] $\hat{P}_{n_{k}}$ converges weakly to some $\hat{P}$;
\item[(iii)] ${\tilde{u}_{n_{k}}}\rightarrow \tilde{u}$ in $\mathcal{X}$   $ \tilde{\mathbb{P}}-a.s.$ and $\tilde{u}$ has the law $\hat{P}\in \mathcal{P}(C([0,T]; H^{-1}))$.
\end{itemize}
\begin{Remark}
Since $\tilde{u}_{n_{k}}$ has the same law as $u_{n_{k}}$,
we immediately have:
$$E^{\tilde{\mathbb{P}}}(\sup\limits_{t\in[0,T]}\| \tilde{u}(t) \|_{L^{2}}^{2})+E^{\tilde{\mathbb{P}}}\int_{0}^{T}\|  \tilde{u}(t) \|_{H^{1,0}}^{2}dt\leq C(1+E^{\tilde{\mathbb{P}}}\| u_{0} \|_{L^{2}}^{2}).$$
\begin{equation}\label{S4eq23}
E^{\hat{P}}(\sup\limits_{t\in[0,T]}\| u(t) \|_{L^{2}}^{2})+E^{\hat{P}}\int_{0}^{T}\| u(t) \|_{H^{1,0}}^{2}dt\leq C(1+E^{\hat{P}}\| u_{0}\|_{L^{2}}^{2}).
\end{equation}
Similarly for $L^{4}(\Omega)$ estimates, we have:
\begin{equation*}
E^{\tilde{\mathbb{P}}}(\sup\limits_{t\in[0,T]}\| \tilde{u}(t)\|_{L^{2}}^{4})+E^{\tilde{\mathbb{P}}}\int_{0}^{T}\| \tilde{u}(t)\|_{L^{2}}^{2}\| \tilde{u}(t)\|_{H^{1,0}}^{2}dt \leq C(1+E^{\tilde{\mathbb{P}}}\| u_{0} \|_{L^{2}}^{4}).
\end{equation*}
\begin{equation}\label{S4eq46}
E^{\hat{P}}(\sup\limits_{t\in[0,T]}\| u(t)\|_{L^{2}}^{4})+E^{\hat{P}}\int_{0}^{T}\| u(t)\|_{L^{2}}^{2}\| u(t)\|_{H^{1,0}}^{2}dt \leq C(1+E^{\hat{P}}\| u_{0} \|_{L^{2}}^{4}).
\end{equation}
\end{Remark}

\subsection{Pass to the Limit and the proof of main theorems.}
In this section we pass the limit as $n\rightarrow \infty$.

 \begin{proof}[Proof of Theorem \ref{thm2}]
 Let us prove $\hat{P}$ satisfies (M1),(M2) and (M3).\\
 (M3) is satisfied by \eqref{S4eq23} .\\
 For (M1), noting that $u_{n}(0)\rightarrow u_{0}\text{  } in \text{  } H  $, we have:
 $$\hat{P}(u(0)=u_{0})=\tilde{\mathbb{P}}(\tilde{u}(0)=u_{0})=\lim\limits_{n\rightarrow \infty } \tilde{\mathbb{P}}(\tilde{u}_{n}(0)=P_{n}u_{0})=1, $$
\begin{equation*}
\begin{split}
 &\hat{P}\Bigl\{u\in C([0,T],H^{-1}):\int_{0}^{T}\| F(u(s))\|_{H^{-1}}ds+\int_{0}^{T}\| \sigma(s,u(s))\|_{L^{2}(\ell^2,H)}^{2}ds<+\infty\Bigr\}\\
 &=\tilde{\mathbb{P}}
 \Bigl\{\tilde{u}\in C([0,T], H^{-1}):\int_{0}^{T}\| F(\tilde{u}(s))\|_{H^{-1}}ds+\int_{0}^{T}\| \sigma(s,\tilde{u}(s))\|_{L^{2}(\ell^2,H)}^{2}ds<+\infty\ \Bigr\}.
 \end{split}
 \end{equation*}
 Since $$\tilde{u}_{n}\rightarrow \tilde{u}\text{ }  \text{in}\text{ }  \mathcal{X} \text{ } \tilde{\mathbb{P}}- a.s. ,$$
 $$\tilde{u}\in L^{2}([0,T], H^{1,1})\cap L^{\infty}([0,T],H^{0,1}) \text{ }  \tilde{\mathbb{P}}-a.s..$$
  Thus by the growth condition of $\sigma$, we have\\
 $\int_{0}^{T}\| \sigma(s,\tilde{u}(s)) \|_{L^{2}(\ell^2,H)}^{2}ds\leq
 \int_{0}^{T}(K_{0}+K_{1}\| \tilde{u} \|_{L^{2}}^{2}+ K_{2}\| \partial_{1}\tilde{u} \|_{L^{2}}^{2})ds<\infty$,   $ \tilde{\mathbb{P}}-a.s.$

Again, we know by interpolation, $\tilde{u}\in L^{4}([0,T],H^{\frac{1}{2}})$ $\tilde{\mathbb{P}}-a.s.$ By Sobolev imbedding, $\tilde{u}\in L^{4}([0,T],{L^{4}})$ $\tilde{\mathbb{P}}-a.s.$ Hence $\dive(\tilde{u}\otimes \tilde{u})\in L^{2}([0,T],H^{-1})$ $\tilde{\mathbb{P}}-a.s.$ And $\partial_{1}^{2}\tilde{u}\in
 L^{1}([0,T],H^{-1})$, $F(\tilde{u}(s))\in L^{1}([0,T],H^{-1})$ $\tilde{\mathbb{P}}-a.s.$
Thus (M1) is satisfied.\\
Finally we prove (M2):\\

Set
\beno
\begin{split}
G^{(1)}(t,u)&:=\langle u(t),l \rangle;\\
G^{(2)}(t,u)&:=\int_{0}^{t}\langle F(u(s)),l \rangle \,ds.
\end{split}
\eeno
Since $\tilde{u}_{n}\rightarrow \tilde{u}$ in $C([0,T],H^{-1})$ and $l\in C^{1}(\mathbb{T}^{2})$, we have for $\tilde{\mathbb{P}}-a.s.$
$$\| \langle \tilde{u}_{n}(t),l \rangle-\langle \tilde{u}(t),l \rangle\|_{L^{\infty}(0,T)} \rightarrow 0                                              $$
Moreover, since $\tilde{u}_{n}$ is bounded in $L^{4}(\Omega, L^{\infty}([0,T],L^{2}))$,
$\langle \tilde{u}_{n}(t),l \rangle$ is bounded in $L^{4}(\Omega)$ for any $t$, we have
$$\lim\limits_{n\rightarrow \infty}E^{\tilde{\mathbb{P}}}\mid G^{(1)}(t,\tilde{u}_{n})-G^{(1)}(t,\tilde{u})\mid=0$$
For $G^{(2)}$, since
 $$ \tilde{u}_{n}\rightarrow \tilde{u} \text{  in  } L^{2}([0,T],L^{2})\text{  } \tilde{\mathbb{P}}-a.s.,$$
 we have
$$\int_{0}^{t}\langle F(\tilde{u}_{n}(s)),l \rangle \,ds  \rightarrow  \int_{0}^{t}\langle F(\tilde{u}(s)),l \rangle\, ds    \ \                         \tilde{\mathbb{P}}-a.s., l\in C^{1}(\mathbb{T}^{2}).$$
Moreover, since $ \tilde{u}_{n}$ is bounded in $L^{4}(\Omega, L^{\infty}([0,T],L^{2}))\cap L^{2}(\Omega, L^{2}([0,T],H^{1})) $ and $l\in C^{1}(\mathbb{T}^{2})$, we have
$$E^{\tilde{\mathbb{P}}}\Big(\int_{0}^{T}\mid \langle F(\tilde{u}_{n}(s)),l \rangle\mid ds\Big)^2\leq C.$$
Therefore, we have
$$\lim\limits_{n\rightarrow \infty}E^{\tilde{\mathbb{P}}}\mid G^{(2)}(t,\tilde{u}_{n})-G^{(2)}(t,\tilde{u})\mid=0.$$
By the definition of $M_{l}$ in (M2), we have
\begin{equation}\label{S4eq24}
\lim\limits_ {n\rightarrow \infty}E^{\tilde{\mathbb{P}}}\mid M_{l}(t,\tilde{u}_{n})-M_{l}(t,\tilde{u})\mid=0.
\end{equation}
Let $t>s$ and $g$ be any bounded and real-valued $\mathcal{F}_{s}$-measurable continuous function on $\mathcal{X}$. Using \eqref{S4eq24} we have:
\begin{equation*}
\begin{split}
E^{\hat{P}}((M_{l}(t,u)-M_{l}(s,u))g(u))&=
E^{\tilde{\mathbb{P}}}((M_{l}(t,\tilde{u})-M_{l}(s,\tilde{u}))g(\tilde{u}))\\
&=\lim\limits_ {n\rightarrow \infty}E^{\tilde{\mathbb{P}}}
((M_{l}(t,\tilde{u}_{n})-M_{l}(t,\tilde{u}_{n}))g(\tilde{u}_{n}))\\
&=\lim\limits_ {n\rightarrow \infty}E^{\hat{P}_{n}}((M_{l}(t,u)-M_{l}(s,u))g(u))\\
&=0,
\end{split}
\end{equation*}
where the last step is due to (M2) for $\hat{P}_{n}$.

Then we have
\begin{equation}\label{S4eq41}
E^{\hat{P}}(M_{l}(t,u)\mid \mathcal{F}_{s})=M_{l}(s,u).
\end{equation}
On the other hand, by Burkholder-Davis-Gundy's inequality, growth condition of $\sigma$ and Lemma 4.1, Lemma 4.2 we have:
\begin{equation*}
\begin{split}
\sup\limits_{n}E^{\tilde{\mathbb{P}}}\mid M_{l}(t,\tilde{u}_{n})\mid^{4}
&\leq C \sup\limits_{n}E^{\tilde{\mathbb{P}}}(\int_{0}^{t}\| \sigma^{*}(\tilde{u}_{n}(s))(l)\|_{\ell^{2}}^{2}\,ds)^{2}\\
&\leq C \sup\limits_{n}\int_{0}^{t}E^{\tilde{\mathbb{P}}}(\| \sigma^{*}(\tilde{u}_{n}(s))(l)\|_{\ell^{2}}^{4})\,ds\\
&\leq C \sup\limits_{n}E^{\tilde{\mathbb{P}}}\int_{0}^t\|\sigma^{*}(s,\tilde{u}_{n})\|_{L^2(H^{1},\ell^{2})}^4
ds\|l\|_{H^1}^4\\
&= C \sup\limits_{n}E^{\tilde{\mathbb{P}}}\int_{0}^t\|\sigma(s,\tilde{u}_{n})\|_{L^2(\ell^{2},H^{-1})}^4
ds\|l\|_{H^1}^4\\
&\leq C\sup\limits_{n}E^{\tilde{\mathbb{P}}}\int_{0}^t(K'_{0}+K'_{1}\| \tilde{u}_{n} \|_{L^{2}}^{2})^{2}
ds\\
&<+\infty,
\end{split}
\end{equation*}
where the third inequality is due to the reason that the normal norm of the operator
is smaller than the Hilbert-Schmidt norm of the operator, the fourth inequality is
a result of $\|\sigma\|_{L^2(\ell^{2},H^{-1})}=\|\sigma^{*}\|_{L^2(H^1,\ell^{2})}$
.\\
Then by \eqref{S4eq24} we obtain
$$\lim\limits_{n\rightarrow \infty}E^{\tilde{\mathbb{P}}}\mid M_{l}(t,\tilde{u}_{n})- M_{l}(t,\tilde{u})\mid^{2}=0.$$
On the other hand, by Lipchitz condition of $\sigma$,
$$\lim\limits_{n\rightarrow \infty}E^{\tilde{\mathbb{P}}} \int_{0}^{t}\|
\sigma^{*}(s,\tilde{u}_{n}(s))(l)-\sigma^{*}(s,\tilde{u}(s))(l)\|_{\ell^{2}}^{2}
ds=0.$$
Thus, using the same method used for proving $E^{\hat{P}}(M_{l}(t,u)\mid \mathcal{F}_{s})=M_{l}(s,u)$, we obtain
$$E^{\hat{P}}(M_{l}^{2}(t,u)-\int_{0}^{t}\|\sigma^{*}(s,u(s))(l)\|
_{l^{2}}^{2}ds\mid \mathcal{F}_{s})=M_{l}^{2}(s,u)-\int_{0}^{s} \|\sigma^{*}(r,u(r))(l)\|
_{l^{2}}^{2}dr.$$
(M2) holds.\\
The results follow.
\end{proof}\\

Finally let us turn to the proof of  the pathwise uniqueness.

 \begin{proof}[Proof of Theorem \ref{thm3}]
Set
$$\tilde{w}:=u-v.$$
Then we have
\begin{equation}\label{S4eq25}
\langle u(t),e_{k}\rangle =\langle u(0),e_{k}\rangle +\int_{0}^{t}\langle-u\cdot\nabla u+\partial_{1}^{2}u,e_{k})ds+\int_{0}^{t}\langle \sigma(s,u(s))dW(s),e_{k}\rangle,
\end{equation}
and
\begin{equation}\label{S4eq26}
\langle v(t),e_{k}\rangle =\langle v(0),e_{k}\rangle +\int_{0}^{t}\langle-v\cdot\nabla v+\partial_{1}^{2}v,e_{k})ds+\int_{0}^{t}\langle \sigma(s,v(s))dW(s),e_{k}\rangle .
\end{equation}
\eqref{S4eq25}-\eqref{S4eq26} ensures that
\begin{equation}\label{S4eq27}
\langle \tilde{w}(t),e_{k}\rangle = \int_{0}^{t}\langle-\tilde{w}\cdot\nabla u+v\cdot\nabla \tilde{w}+\partial_{1}^{2}\tilde{w},e_{k}\rangle ds+\int_{0}^{t}\langle \sigma(s,u(s))- \sigma(s,v(s))dW(s),e_{k}\rangle .
\end{equation}
Set $\varphi_{k}:=\langle \tilde{w}(s), e_{k}\rangle$. It\^{o}'s formula and \eqref{S4eq27} yield:
\begin{equation*}
\begin{split}
d\varphi_{k}^{2}&=2\varphi_{k}d\varphi_{k}+\| (\sigma(s,u(s))- \sigma(s,v(s)))^{*}e_{k}\|_{l^{2}}^{2}ds\\
&=2\langle \tilde{w}(s),e_{k}\rangle\langle-\tilde{w}\cdot\nabla u+v\cdot\nabla \tilde{w}+\partial_{1}^{2}\tilde{w},e_{k}\rangle ds\\
&+2\langle \tilde{w}(s),e_{k}\rangle\langle \bigl(\sigma(s,u(s))- \sigma(s,v(s))\bigl)dW(s),e_{k}\rangle\\
&+\| \bigl(\sigma(s,u(s))- \sigma(s,v(s))\bigl)^{*}e_{k}\|_{l^{2}}^{2}ds.
\end{split}
\end{equation*}
Since $L_{2}<\frac{2}{5}$, we can choose $0<\hat{\alpha}<1$ and $0<\hat{\beta}<1$, such that $L_{2}+2\hat{\alpha}+\frac{4}{\hat{\beta}}L_{2}<2$.
From the calculation in the uniqueness part of deterministic cases (the proof of Thm. 3.1),  we know that
\begin{equation}\label{S4eq42}
\begin{split}
&\mid(\tilde{w}\cdot \nabla u\mid \tilde{w})\mid\\
&\leq\| \tilde{w}\cdot \nabla u \|_{L^{2}}\| \tilde{w} \|_{L^{2}}\\
&\leq\hat{\alpha}\| \partial_{1}\tilde{w} \|_{L^{2}}^{2}+C(\hat{\alpha})\bigl(\| \partial_{1}u \|_{L^{2}}^{\frac{2}{3}}\| \partial_{1}\partial_{2}u \|_{L^{2}}^{\frac{2}{3}}
+\| \partial_{2}u \|_{L^{2}}^{\frac{2}{3}}\| \partial_{1}\partial_{2}u \|_{L^{2}}^{\frac{2}{3}}\bigr)
\| \tilde{w} \|_{L^{2}}^{2}.
\end{split}
\end{equation}
Set $q(t):=\int_{0}^{t}2C(\hat{\alpha})(\| \partial_{1}u\|_{L^{2}}^{\frac{2}{3}}+\| \partial_{2}u\|_{L^{2}}^{\frac{2}{3}})\| \partial_{1}\partial_{2}u\|_{L^{2}}^{\frac{2}{3}}ds$\\
By It\^{o}'s formula:
\begin{equation}\label{S4eq28}
\begin{split}
e^{-q(t)}\varphi_{k}^{2}=\varphi (0)^{2}
&+2\int_{0}^{t}e^{-q(s)}\langle \tilde{w}(s),e_{k}\rangle\langle-\tilde{w}\cdot\nabla u+v\cdot\nabla \tilde{w}+\partial_{1}^{2}\tilde{w},e_{k}\rangle ds\\
&+2\int_{0}^{t}e^{-q(s)}\langle \tilde{w}(s),e_{k}\rangle\langle \bigl(\sigma(s,u(s))- \sigma(s,v(s))\bigl)dW(s),e_{k}\rangle\\
&+\int_{0}^{t}e^{-q(s)}\| (\sigma(s,u(s))- \sigma(s,v(s)))^{*}e_{k}\|_{l^{2}}^{2}ds-\int_{0}^{t}q'(s)e^{-q(s)}\varphi_{k}^{2}ds.
\end{split}
\end{equation}
Notice that
\begin{equation}\label{S4eq29}
\sum\limits_{k=1}^{\infty}\varphi_{k}^{2}=\| \tilde{w}(t)\|_{L^{2}}^{2}.
\end{equation}

The dominated convergence theorem imply when $N\rightarrow \infty$:
\begin{equation*}
\begin{split}
&2\mid \sum\limits_{k\leq N}\int_{0}^{t}e^{-q(s)}\langle \tilde{w}(s),e_{k}\rangle\langle-\tilde{w}\cdot\nabla u,e_{k}\rangle ds\mid\\
&\longrightarrow 2\mid\int_{0}^{t}e^{-q(s)}\langle \tilde{w}\cdot\nabla \tilde{w} ,u\rangle ds\mid\\
\end{split}
\end{equation*}
since by \eqref{S4eq42},
\begin{equation}\label{S4eq30}
\begin{split}
 &\int_{0}^{t}\| \tilde{w}\cdot \nabla u \|_{L^{2}}\| \tilde{w} \|_{L^{2}}ds\\
&\leq  \int_{0}^{t}2\hat{\alpha}e^{-q(s)}\| \partial_{1}\tilde{w} \|_{L^{2}}^{2}
+2C(\hat{\alpha})e^{-q(s)}(\| \partial_{1}u \|_{L^{2}}^{\frac{2}{3}}\| \partial_{1}\partial_{2}u \|_{L^{2}}^{\frac{2}{3}}
+\| \partial_{2}u \|_{L^{2}}^{\frac{2}{3}}\| \partial_{1}\partial_{2}u \|_{L^{2}}^{\frac{2}{3}}
)\| \tilde{w} \|_{L^{2}}^{2}ds\\
&=\int_{0}^{t}[2\hat{\alpha}e^{-q(s)}\| \partial_{1}\tilde{w} \|_{L^{2}}^{2}+e^{-q(s)}q'(s)\| \tilde{w} \|_{L^{2}}^{2}]ds.
\end{split}
\end{equation}

Notice that
$$\|  (\sum\limits_{k\leq N} \langle \tilde{w}(s),e_{k}\rangle e_{k})\|_{L^{2}}\|\leq \|  \tilde{w}(s) \|_{L^{2}}.$$
Now we follow the same calculation of \eqref{S3eq30} and \eqref{S3eq31}:
\begin{equation*}
\begin{split}
&\mid e^{-q(s)}\langle v\cdot \nabla \tilde{w} , \sum\limits_{k\leq N} \langle \tilde{w}(s),e_{k}\rangle e_{k}\rangle\mid \\
&\leq  e^{-q(s)}\bigl( \| v \|_{L^{2}}^{\frac{1}{2}}\| \partial_{1}v \|_{L^{2}}^{\frac{1}{2}}
\| \partial_{1}\tilde{w} \|_{L^{2}}^{\frac{1}{2}}\| \partial_{1}\partial_{2}\tilde{w} \|_{L^{2}}^{\frac{1}{2}}
+\| v \|_{L^{2}}^{\frac{1}{2}}\| \partial_{1}v \|_{L^{2}}^{\frac{1}{2}}
\| \partial_{2}\tilde{w} \|_{L^{2}}^{\frac{1}{2}}\| \partial_{1}\partial_{2}\tilde{w} \|_{L^{2}}^{\frac{1}{2}}\bigr)\| w \|_{L^{2}}.
\end{split}
\end{equation*}
Since the latter $\in L^1 ([0,T])$ for $t$,
we use dominated convergence theorem again and obtain:
\begin{equation}\label{S4eq31}
2\sum\limits_{k\leq N}\int_{0}^{t}e^{-q(s)}\langle \tilde{w}(s),e_{k}\rangle\langle v\cdot\nabla \tilde{w},e_{k}\rangle ds \longrightarrow 0 \text{  as $N\rightarrow\infty$ }.
\end{equation}
Similarly, by dominated convergence theorem,
\begin{equation}\label{S4eq32}
2\sum\limits_{k\leq N}\int_{0}^{t}e^{-q(s)}\langle \tilde{w}(s),e_{k}\rangle \langle\partial_{1}^{2}\tilde{w},e_{k}\rangle ds
 \longrightarrow -2\int_{0}^{t}e^{-q(s)}\| \partial_{1}\tilde{w}\|_{L^{2}}^{2}ds
 \text{  as $N\rightarrow\infty$ },
\end{equation}
and
\begin{equation}\label{S4eq33}
\begin{split}
&\sum\limits_{k\leq N}\int_{0}^{t}e^{-q(s)}\| (\sigma(s,u(s))-\sigma(s,v(s)))^{*}e_{k}\|_{l^{2}}^{2}ds\\
&\longrightarrow \int_{0}^{t}e^{-q(s)}\| \sigma(s,u(s))-\sigma(s,v(s))\|_{L^{2}(\ell^2,H)}^{2}ds\\
&\leq \int_{0}^{t}e^{-q(s)}(L_{1}\| \tilde{w}\|_{L^{2}}^{2}+L_{2}\|
\partial_{1}\tilde{w}\|_{L^{2}}^{2} )ds.
\end{split}
\end{equation}
Since $L_{2}<\frac{2}{5}$, we can choose $0<\hat{\beta}<1$, such that $\frac{4}{\hat{\beta}}L_{2}+L_{2}<2$.
By Burkholder-Davis-Gundy's inequality as well as the dominated convergence theorem, we deduce
\begin{equation}\label{S4eq34}
\begin{split}
&2\mid E(\sup\limits_{0\leq s \leq t} \int_{0}^{s}\sum\limits_{k\leq N} e^{-q(r)}\langle \tilde{w}(r),e_{k}\rangle\langle (\sigma(r,u(r))- \sigma(r,v(r)))dW(r),e_{k}\rangle)\mid \\
&\leq 4 E\Bigl(\int_{0}^{t}e^{-2q(s)}\|\sigma(s,u(s))- \sigma(s,v(s))\|_{L^2(\ell^2,H)}^2\| \tilde{w}\|_{L^{2}}^2 \Bigl)^{\frac{1}{2}}\\
&\leq \hat{\beta}E(\sup\limits_{s\leq t}(e^{-q(s)} \| \tilde{w}(s) \|_{L^{2}}^{2}))\\
&+\frac{4}{\hat{\beta}}E\int_{0}^{t}e^{-q(s)}(L_{1}\| \tilde{w}(s) \|_{L^{2}}^{2}+L_{2}\| \partial_{1}\tilde{w}(s)\|_{L^{2}}^{2})ds.
\end{split}
\end{equation}
Combining \eqref{S4eq28}-\eqref{S4eq34} and dropping some negative terms, we obtain:
$$(1-\hat{\beta})E(\sup\limits_{0\leq s \leq t}e^{-q(t)}\| \tilde{w}\|_{L^{2}}^{2})\leq E\int_{0}^{t}(1+\frac{4}{\hat{\beta}})L_{1}e^{-q(s)}\| \tilde{w}(s)\|_{L^{2}}^{2}ds.$$
By Gronwall's inequality we obtain $\tilde{w}=0$  $\tilde{\mathbb{P}}-a.s.$
\end{proof}

\bigbreak \noindent {\bf Acknowledgments.}  We would like to thank Professor Zhiming Ma for constant  encouragement and profitable guidance.
  P. Zhang is partially supported
    by NSF of China under Grants   11371347 and 11688101,  and innovation grant from National Center for
    Mathematics and Interdisciplinary Sciences of the Chinese
Academy of Sciences .
R. Zhu is supported in part by NSFC (11771037, 11671035, 11371099). Financial support by the DFG through the CRC 1283 "Taming uncertainty and profiting from randomness and low regularity in analysis, stochastics and their applications" is acknowledged.

\nocite{*}

\end{document}